\RequirePackage{fix-cm}
\RequirePackage{amsthm}
\documentclass[sn-mathphys-num,pdflatex]{sn-jnl}

\usepackage{graphicx}
\usepackage{multirow}
\usepackage{amsmath,amssymb,amsfonts}
\usepackage{mathrsfs}
\usepackage[title]{appendix}
\usepackage{xcolor}
\usepackage{textcomp}
\usepackage{manyfoot}
\usepackage{booktabs}
\usepackage{algorithm}
\usepackage{algorithmicx}
\usepackage{algpseudocode}
\usepackage{listings}
\usepackage{anyfontsize}

\usepackage{hyperref}
\usepackage{lmodern}
\usepackage{enumerate} 
\usepackage{enumitem}

\newtheorem{theorem}{Theorem}[section]
\newtheorem{definition}{Definition}[section]
\newtheorem{lemma}{Lemma}[section]
\newtheorem{corollary}{Corollary}[section]
\newtheorem{example}{Example}[section]
\newtheorem{remark}{Remark}[section]

\newcounter{assumcounter}
\newtheorem{assum}[assumcounter]{Assumption}

\newcounter{oraclecounter}
\newtheorem{oracle}[oraclecounter]{Oracle}

\newcommand{\eps}{\varepsilon}
\newcommand{\R}{\mathbb{R}}
\newcommand{\N}{\mathbb{N}}
\newcommand{\C}{\mathcal{C}}
\newcommand{\Bcl}{\bar{B}}
\newcommand{\lc}[1]{lower-$\C^{#1}$}
\newcommand{\calM}{\mathcal{M}}
\newcommand{\calR}{\mathcal{R}}
\newcommand{\T}{\mathcal{T}}
\newcommand{\deriv}{D}

\newcommand{\algApproxW}{Alg.\ \ref{algo:approx_W}}
\newcommand{\algGlobal}{Alg.\ \ref{algo:global_method}}
\newcommand{\ipopt}{\texttt{IPOPT}}
\newcommand{\mexipopt}{\texttt{mexIPOPT}}

\DeclareMathOperator*{\conv}{conv}
\DeclareMathOperator*{\rb}{rb}
\DeclareMathOperator*{\aff}{aff}
\DeclareMathOperator*{\pr}{pr}
\DeclareMathOperator*{\argmin}{arg\,min}
\DeclareMathOperator*{\vspan}{span}

\begin{document}

\title{Enclosing minima in nonsmooth optimization via trust regions of higher-order cutting-plane models}

\author*[1]{\fnm{Bennet} \sur{Gebken}}\email{bennet.gebken@tum.de}

\author[1]{\fnm{Michael} \sur{Ulbrich}}\email{mulbrich@ma.tum.de}

\affil[1]{\orgdiv{Department of Mathematics}, \orgname{Technical University of Munich}, \orgaddress{\street{Boltzmannstr. 3}, \city{Garching b. München}, \postcode{85748}, \country{Germany}}}

\abstract{
    We propose a globally convergent trust-region bundle method for minimizing \lc{2} functions using higher-order cutting-plane models. Under certain growth assumptions on the objective around its minimum, the method is able to compute infinitely many trust regions of decreasing size that contain the minimum. We show that these growth assumptions are satisfied for certain finite max-type functions with sharp or quadratic growth. Enclosing the minimum in this way can be used to initialize local superlinearly convergent methods, which we demonstrate in numerical experiments.
}


\maketitle

\section{Introduction} \label{sec:introduction}

Given a nonsmooth function $f$ that has a minimum $x^*$, the goal of this work is to construct a method that is able to compute a sequence of balls that contain $x^*$ with a radius that vanishes in the limit. To achieve this, we construct a globally convergent trust-region method that generates a sequence of iterates $(x^j)_j$ such that for a vanishing sequence of trust-region radii $(\Delta_j)_j$, $x^*$ lies inside the closed $\Delta_j$-ball around $x^j$ (i.e., inside the trust region) for infinitely many $j \in \N$. The main motivation of this work is the globalization of the trust-region bundle method of \cite{GU2026a} (written in parallel to the current work), which achieves local R-superlinear convergence (of serious steps) for nonsmooth optimization problems with \lc{2} functions \cite{RW1998}. To prove the local convergence, it had to be assumed that the initial trust-region radius is small enough and that the initial trust region contains the minimum $x^*$. By applying the local method with initial point $x^j$ and initial radius $\Delta_j$ for each $j \in \N$ within the global method, it can be guaranteed that eventually, there is an application of the local method that converges R-superlinearly to the minimum.

While the main motivation of our method is the globalization of the method from \cite{GU2026a}, its properties are also interesting in a broader context: In general, superlinearly convergent methods for black-box nonsmooth optimization problems, where no additional structure of the objective function $f : \R^n \rightarrow \R$ is explicitly available, are difficult to construct \cite{MS2012}. One of the challenges is that building an accurate local model of a nonsmooth function around its minimum $x^*$ requires more than a single point close to $x^*$. For example, the $\mathcal{V}\mathcal{U}$\emph{-algorithm} \cite{LOS1999,MS2005} requires the automatic identification of the $\mathcal{V}$- and $\mathcal{U}$-spaces at $x^*$, which, in general, requires knowledge of the entire Clarke subdifferential at $x^*$ (cf.\ \cite{MS2005}, Sec.\ 2.1). For the standard oracle, which only provides a single subgradient at each point, approximating these spaces thus requires more than a single point close to $x^*$. The $k$\emph{-bundle Newton method} \cite{LW2019} requires an initial bundle of points that are close enough to $x^*$ and each lie in a different ``smooth region'' of $f$ around $x^*$ (cf.\ \cite{LW2019}, Thm.\ 5.11). Finally, the method \emph{SuperPolyak} \cite{CD2024} for minimizing functions with a sharp minimum (and with known optimal value $f(x^*)$) is based on generating multiple points close enough around $x^*$ with linearly independent subgradients (cf.\ \cite{CD2024}, p.\ 1679 and Lem.\ 10). So in the general context of superlinearly convergent methods for nonsmooth optimization problems, our method can be seen as a way to obtain arbitrarily small neighborhoods of $x^*$ in which the required local information of $f$ can be obtained. 

The theory of our method is based on the quantity
\begin{align} \label{eq:def_z_star_Lambda}
	\Lambda^p(x,\Delta) := \frac{f(x) - f(z^*(x,\Delta))}{\Delta^p} \geq 0,
	\ \text{where} \
	z^*(x,\Delta) \in \argmin_{z \in \Bcl_\Delta(x)} f(x),
\end{align}
$x \in \R^n$, $\Delta > 0$, $\Bcl_\Delta(x) := \{ y \in \R^n : \| y - x \| \leq \Delta \}$, $\| \cdot \|$ is the Euclidean norm and $p \in \N$. In words, $\Lambda^p(x,\Delta)$ is the best possible decrease of $f$ in the trust region $\Bcl_\Delta(x)$ relative to the trust-region radius $\Delta$ to the power $p$. 
To obtain a tool for showing that $x^*$ lies inside the trust region $\Bcl_\Delta(x)$, we consider the following property:
\begin{equation} \label{eq:property_P} \tag{P}
    \begin{aligned}
        &\text{For } p \in \N \text{ and } x^* \in \R^n, \text{ there are an open neighborhood } U \subseteq \R^n \\
        &\text{of } x^* \text{ and a constant } C > 0 \text{ such that } \Lambda^p(x,\Delta) \geq C \text{ for all } x \in U \\ 
        &\text{and } \Delta > 0 \text{ with } x^* \notin \Bcl_\Delta(x).
    \end{aligned}
\end{equation}
If the property \eqref{eq:property_P} holds, then as long as $x^* \notin \Bcl_\Delta(x)$, the value of $f$ can be reduced by at least $C \Delta^p$. Furthermore, if $\lim_{j \rightarrow \infty} \Lambda^p(x^j,\Delta_j) = 0$ for sequences $(x^j)_j$ and $(\Delta_j)_j$ with $x^j \rightarrow x^*$ and $\Delta_j \rightarrow 0$, then we have the desired inclusion $x^* \in \Bcl_{\Delta_j}(x^j)$ for all $j$ large enough. We will show that \eqref{eq:property_P} is related to $p$\emph{-order growth} of $f$ around $x^*$, which means that there are an open neighborhood $U \subseteq \R^n$ of $x^*$ and some constant $\beta > 0$ such that 
\begin{align} \label{eq:def_p_order_growth}
    f(x) \geq f(x^*) + \beta \| x - x^* \|^p \quad \forall x \in U.
\end{align}
More precisely, we derive sufficient conditions for \eqref{eq:property_P} to hold for \lc{2} functions with \emph{sharp} ($p = 1$) or \emph{quadratic} ($p = 2$) growth. While the quantity $\Lambda^p(x,\Delta)$ is purely theoretical, we can approximate it in practice by using the higher-order cutting-plane models (of order $q \geq p$) introduced in \cite{GU2026a}. This allows for the construction of a trust-region method that is globally convergent to critical points for general \lc{2} functions, and where $x^*$ lies inside the trust region infinitely many times if \eqref{eq:property_P} holds. While several globally convergent trust-region methods for black-box nonsmooth optimization problems have already been proposed \cite{SZ1992,HUL1993b,ANR2015,M2023}, to the best of the authors' knowledge, the question whether they provably generate vanishing trust regions containing the minimum has not been considered before.

The remainder of this work is structured as follows: In Sec.\ \ref{sec:preliminaries} we introduce the notation and the basics that we need, including the relevant results from \cite{GU2026a}. In Sec.\ \ref{sec:algorithm_and_convergence} we derive our method, prove global convergence (in the sense of criticality of its accumulation points) for \lc{2} functions, and prove that infinitely many trust regions contain $x^*$ if \eqref{eq:property_P} holds. Sec.\ \ref{sec:property_P_for_max_type} is concerned with proving that \eqref{eq:property_P} holds for the case where the objective is the maximum of finitely many smooth functions with sharp or quadratic growth. In Sec.\ \ref{sec:numerical_experiments} we verify the properties of our method in numerical experiments and combine it with the methods from \cite{GU2026a} and \cite{LW2019} to obtain an overall globally convergent method with local R-superlinear convergence. Finally, in Sec.\ \ref{sec:discussion_and_outlook}, we discuss our results and directions for future research.

A Matlab implementation of our method, including scripts for the reproduction of all experiments shown in this work, is available at \url{https://github.com/b-gebken/higher-order-trust-region-bundle-method}.

\section{Preliminaries} \label{sec:preliminaries}

In this section, we first introduce our class of objective functions and some basics of nonsmooth analysis. To have a self-contained work, we then introduce the higher-order cutting-plane models from \cite{GU2026a}, including all results we will use here. 

\subsection{Nonsmooth analysis} \label{subsec:prelim_nonsmooth_analysis}

    The objective functions we consider are functions that can be written as the maximum of smooth functions. For the purpose of estimating the remainder in $q$-order Taylor expansions later on, we assume that these smooth functions are at least $\C^{q+1}$ (i.e., $(q+1)$-times continuously differentiable). More formally, for $q \in \N$, we assume there is a global \lc{q+1} representation (see, e.g., \cite{RW1998}, Def.\ 10.29) in the following sense:
     
    \begin{assum} \label{assum:B1}
        For $q \in \N$ there are a compact topological space $S$ and $\C^{q+1}$ functions $f_s : \R^n \rightarrow \R$, $s \in S$, such that 
        \begin{align*}
            f(x) = \max_{s \in S} f_s(x) \quad \forall x \in \R^n
        \end{align*}
        and $f_s$ and all its partial derivatives up to order $q+1$ depend continuously on $(s,x) \in S \times \R^n$. 
    \end{assum}   
     
	For a function $f$ satisfying \ref{assum:B1}, the \emph{active set} at a point $x \in \R^n$ is the set $A(x) := \{ s \in S : f_s(x) = f(x) \}$. The functions $f_s$, $s \in S$, are called \emph{selection functions}, and a selection function $f_s$ is called \emph{active} at $x$ if $s \in A(x)$. By \cite{RW1998}, Thm.\ 10.31, $f$ is locally Lipschitz continuous with its Clarke subdifferential \cite{C1990} given by $\partial f(x) = \conv(\{ \nabla f_s(x) : s \in A(x) \})$, where $\conv(\cdot)$ denotes the convex hull. For a locally Lipschitz continuous $f$, we say that a point $x$ is a \emph{critical point} of $f$ if $0 \in \partial f(x)$. For functions satisfying \ref{assum:B1}, this means that there is a convex combination of the gradients of active selection functions at $x$ which is zero. If $S$ is finite, then $f$ is called a \emph{finite max-type function}. The only difference of \ref{assum:B1} to the definition of \lc{q+1} functions is the fact that the representation in \ref{assum:B1} has to hold globally for any $x \in \R^n$. By \cite{RW1998}, Cor.\ 10.34 and Prop.\ 10.54, any \lc{2} function has a representation as in \ref{assum:B1} on any bounded subset of $\R^n$ and for any $q \in \N$. Since the method we derive in this work always generates bounded sequences (for functions with bounded sublevel sets), assuming a global representation as in \ref{assum:B1} for a \lc{2} function is thus no practical restriction, and merely simplifies the presentation of our results.
    
    For the proof of global convergence of our method, we will make use of the \emph{Goldstein} $\eps$\emph{-subdifferential} \cite{G1977}. For $x \in \R^n$ and $\eps \geq 0$, it can be defined as
	\begin{align*}
		\partial_\eps f(x) := \conv \left( \bigcup_{y \in \Bcl_\eps(x)} \partial f(y) \right).
	\end{align*}	 
	Let $g_\eps(x)$ be the element in $\partial_\eps f(x)$ with the smallest norm. Then either $g_\eps(x) = 0$ or 
    \begin{align*}
         f \left( x - \eps \frac{g_\eps(x)}{\| g_\eps(x) \|}  \right) - f(x) \leq -\eps \| g_\eps(x) \|
    \end{align*}
    (see, e.g., Sec.\ 2 in \cite{G2024b}). In particular, for $\Lambda^1$ as in \eqref{eq:def_z_star_Lambda} and any $\Delta > 0$, it holds
    \begin{align} \label{eq:lambda_goldstein_estimate}
        \Lambda^1(x,\Delta)
        \geq \frac{f(x) - f(x - \Delta \frac{g_\Delta(x)}{\| g_\Delta(x) \|})}{\Delta}
        \geq \frac{\Delta \| g_\Delta(x) \|}{\Delta}
        = \| g_\Delta(x) \|.
    \end{align}
    If $(x^j)_j \subseteq \R^n$ and $(\Delta_j)_j \in \R^{>0}$ are sequences with $x^j \rightarrow \bar{x} \in \R^n$, $\Delta_j \rightarrow 0$, and $g_{\Delta_j}(x^j) \rightarrow 0$, then $\bar{x}$ is critical (see, e.g., \cite{G2022}, Lem.\ 4.4.4). This means that $\Lambda^1(x^j,\Delta_j)$ vanishing can be used to show convergence to critical points, which will be our strategy for proving convergence in Sec.\ \ref{sec:algorithm_and_convergence}.

\subsection{Higher-order cutting-plane models}

	Before introducing the higher-order cutting-plane models from \cite{GU2026a}, we first introduce the notation for higher-order Taylor expansion (from \cite{B1964}) and our oracle assumption on the availability of derivatives. To this end, for a selection function $f_s$ and $m \in \{ 1, \dots, q \}$, $y, z, v \in \R^n$, define $\deriv^{0} f_s(y)(v)^0 := f_s(y)$ and
	\begin{equation*}
		\begin{aligned}
			\deriv^{m} f_s(y)(v)^m &:= \sum_{i_1 = 1}^n \cdots \sum_{i_m = 1}^n \partial_{i_1} \cdots \partial_{i_m} f_s(y) v_{i_1} \cdots v_{i_m}, \\
			T^q f_s(z,y) &:= \sum_{m = 0}^q \frac{1}{m!} \deriv^{m} f_s(y)(z - y)^m.
		\end{aligned}
	\end{equation*}
	(Then $\deriv^{1} f_s(y)(v)^1 = \nabla f_s(y)^\top v$ and $\deriv^{2} f_s(y)(v)^2 = v^\top \nabla^2 f_s(y) v$.) Since $f_s$ is $\C^{q+1}$, Taylor's theorem (see, e.g., \cite{B1964}, Thm.\ 20.16) shows that for any $y, z \in \R^n$, there is some $a \in \conv(\{ y, z \})$ such that $f_s(z) = T^q f_s(z,y) + R^{q+1}(z,y)$ for
	\begin{align} \label{eq:Taylor_remainder}
		R^{q+1}(z,y) = \frac{1}{(q+1)!} \deriv^{q+1} f_s(a)(z - y)^{q+1}.
	\end{align}
    At several points throughout this work, we will use the fact that continuity of the partial derivatives of $f_s$ of order $q+1$ implies that for any bounded, convex set $V$, there is a constant $K > 0$ such that $|R^{q+1}(z,y)| \leq K \| z - y \|^{q+1}$ for all $y, z \in V$. Concerning the availability of higher-order derivatives for our method, we make the following oracle assumption:
     
	\begin{oracle} \label{oracle:1}
		For a function $f : \R^n \rightarrow \R$ satisfying \ref{assum:B1} and for each $x \in \R^n$, we have access to the objective value $f(x)$ and the maps $v \mapsto \deriv^{m} f_{s(x)}(x)(v)^m$ for some $s(x) \in A(x)$ and all $m \in \{1, \dots, q\}$.
	\end{oracle}
     
	Note that we only assume that we have access to the derivatives of a selection function $f_{s(x)}$, but not to the index $s(x)$ itself or any other information of $f_{s(x)}$. Further details of this oracle are discussed in \cite{GU2026a}, Sec.\ 3 and 7.

	Now let $W \subseteq \R^n$ be a nonempty, finite set and assume that $f$ satisfies \ref{assum:B1}. The $q$\emph{-order cutting-plane model} of $f$ for the ``bundle'' $W$ is defined as  
	\begin{align*}
		\T^{q,W}(z) 
		:= \max_{y \in W} T^q f_{s(y)}(z,y)
		= \max_{y \in W} \sum_{m = 0}^q \frac{1}{m!} \deriv^{m} f_{s(y)}(y)(z - y)^m.
	\end{align*}
    For $q = 1$ this model is a standard cutting-plane model. A visualization for larger $q$ is shown in \cite{GU2026a}, Fig.\ 1. Due to the piecewise nature of $f$, an error estimate for this model can only be obtained with respect to the selection functions that have been captured via $W$. More formally, for $s(W) := \{ s(y) : y \in W \}$ consider
    \begin{align*}
        f^W(z) := \max_{y \in W} f_{s(y)}(z), \quad \calR^{q,W}(z) := f^W(z) - \T^{q,W}(z).
    \end{align*}
    Clearly, if $s(W) = S$ then $f^W = f$ and $\calR^{q,W}$ is the error with respect to the actual objective function. In general, the following estimate for $\calR^{q,W}$ on a trust region $\Bcl_\Delta(x)$ holds (cf.\ \cite{GU2026a}, Lem.\ 3.1):
     
    \begin{lemma} \label{lem:q_order_error_estimate}
        Let $q \in \N$ and assume that $f$ satisfies \ref{assum:B1}. Then for every bounded set $V \subseteq \R^n$ and every $\hat{\Delta} > 0$ there is some $K \geq 0$ such that
        \begin{align*}
            \max_{z \in \Bcl_\Delta(x)} |\calR^{q,W}(z)| \leq K \Delta^{q+1}
        \end{align*}
        for all $x \in V$, $\Delta \in [0,\hat{\Delta}]$, and finite, nonempty sets $W \subseteq \Bcl_\Delta(x)$.
    \end{lemma}
     
    For $x \in \R^n$ and $\Delta > 0$, the trust-region subproblem induced by the higher-order cutting-plane model is
    \begin{align} \label{eq:def_bar_z}
        \min_{z \in \Bcl_\Delta(x)} \T^{q,W}(z), 
    \end{align}
    which can be rewritten as the smooth, constrained problem
    \begin{equation} \label{eq:bar_z_epigraph}
        \begin{aligned}
            \min_{z \in \R^n, \theta \in \R} \ & \theta \\
            \text{s.t.} \ & T^q f_{s(y)}(z,y) \leq \theta \quad \forall y \in W,\\
            & \| z - x \|^2 \leq \Delta^2.
        \end{aligned}
    \end{equation}
    Let $\bar{z}^{q,W}(x,\Delta)$ be a solution of \eqref{eq:def_bar_z}. For the sake of brevity, we write $\bar{z}^W = \bar{z}^{q,W}(x,\Delta)$ whenever the context allows. Clearly, the set $W$ has to be ``rich enough'' to have a chance to have $f(\bar{z}^{q,W}(x,\Delta)) < f(x)$. In \cite{GU2026a}, $W$ was computed using \algApproxW{}, which is a subroutine that iteratively adds new elements to $W$.
    \begin{algorithm} 
        \caption{Compute $W$}
        \label{algo:approx_W}
        \begin{algorithmic}[1] 
            \Require Oracle \ref{oracle:1}, point $x \in \R^n$, radius $\Delta > 0$, $q \in \N$, finite, nonempty set $W^1 \subseteq \Bcl_\Delta(x)$, $\sigma \in (0,1)$, threshold $c  > 0$.
            \For{$i = 1, 2, \dots$}
                \State Compute $\bar{z}^{W^i} = \bar{z}^{q,W^i}(x,\Delta)$ (cf.\ \eqref{eq:def_bar_z}, \eqref{eq:bar_z_epigraph}). \label{state:approx_W_solve_subproblem}
                \If{$f(\bar{z}^{W^i}) - \T^{q,W^i}(\bar{z}^{W^i}) \leq \min(\Delta^{q + \sigma},c)$} \label{state:approx_W_stopping_criterion}
                    \State Stop.
                \Else
                    \State Set $W^{i+1} = W^i \cup \{ \bar{z}^{W^i} \}$.
                \EndIf
            \EndFor
        \end{algorithmic}
    \end{algorithm}
    It is based on the observation that for all $y \in W$ we have $f(y) - \T^{q,W}(y) \leq 0$. This implies that if $f(\bar{z}^W) - \T^{q,W}(\bar{z}^W) > 0$, then $\bar{z}^W \notin W$, so adding $\bar{z}^W$ to $W$ leads to an augmented model.
    To have finite termination, in its original form in \cite{GU2026a}, the inequality $f(\bar{z}^W) - \T^{q,W}(\bar{z}^W) \leq \Delta^{q + \sigma}$ for $\sigma \in (0,1)$ is used as a stopping criterion. In the current work, we use the slightly modified term $\min(\Delta^{q + \sigma},c)$ with a threshold $c > 0$ for the stopping criterion, which allows us to compute more accurate models when $\Delta$ is large. Without having to change the proof, the following behavior of \algApproxW{} can be shown (cf.\ \cite{GU2026a}, Lem.\ 4.1):
     
    \begin{lemma} \label{lem:algo_approx_W_termination}
        Let $q \in \N$ and assume that $f$ satisfies \ref{assum:B1}. Let $x \in \R^n$.
        \begin{enumerate}[label=(\alph*)]
            \item \algApproxW{} terminates.
            \item If $S$ is finite, then there is some $\hat{\Delta} > 0$ such that for all $\Delta \in (0,\hat{\Delta}]$, \algApproxW{} terminates in at most $|S|$ iterations.
        \end{enumerate}
    \end{lemma}

\section{Algorithm and global convergence} \label{sec:algorithm_and_convergence}

In this section, we derive our method and prove its global convergence. (As in \cite{NW2006}, p.\ 40, ``global convergence'' refers to criticality of accumulation points.) The general idea is to use the subproblem \eqref{eq:def_bar_z} in a trust-region framework to generate a sequence of points with (sufficiently) decreasing objective values. Due to the local nature of the Taylor expansions in the model $\T^{q,W}$, this requires a mechanism for controlling the trust-region radius $\Delta$, which we derive in the following. We begin by introducing our notation. For simplicity, we use a predetermined, vanishing sequence $(\Delta_j)_j \subseteq \R^{>0}$ for the consecutive values of the trust-region radius. For each $j \in \N$ let $N_j \in \N$ be the number of steps we perform with the same trust-region radius $\Delta_j$. For ease of notation, we denote the corresponding iterates by $x^{j,i}$, where the index $j$ indicates which radius is currently used and the index $i$ enumerates over the iterates for this radius. More formally, for some initial point $x^{1,0} \in \R^n$, we have
\begin{align*}
    x^{j,i+1} = \bar{z}^{q,W_{j,i}}(x^{j,i},\Delta_j)
    \quad \forall j \in \N, i \in \{0,\dots,N_j - 1\},
\end{align*}
where $x^{j+1,0} = x^{j, N_{j}}$ and $W_{j,i} \subseteq \Bcl_{\Delta_j}(x^{j,i})$.
We have to decide when to change the radius $\Delta_j$, i.e., how many steps $N_j$ to perform for each $j$, to assure that the objective value decreases from $x^{j,i}$ to $x^{j,i+1}$ and that the decrease is sufficient to achieve global convergence. To this end, denote $\bar{z}^{j,i} := \bar{z}^{q,W_{j,i}}(x^{j,i},\Delta_j)$ and $z^*_{j,i} := z^*(x^{j,i},\Delta_j)$ (cf.\ \eqref{eq:def_z_star_Lambda}). By Lem.\ \ref{lem:q_order_error_estimate}, if all $x^{j,i}$ lie in a bounded set, then there is some $K > 0$ such that the best possible point $z^*_{j,i}$ in the trust region and the solution $\bar{z}^{j,i}$ of the subproblem are related via
\begin{equation} \label{eq:deriv_2}
    \begin{aligned}
        f(z^*_{j,i}) 
        &\geq f^{W_{j,i}}(z^*_{j,i})
        = \T^{q,W_{j,i}}(z^*_{j,i}) + \calR^{q,W_{j,i}}(z^*_{j,i}) \\
        &\geq \T^{q,W_{j,i}}(\bar{z}^{j,i}) + \calR^{q,W_{j,i}}(z^*_{j,i}) \\
        &= f(\bar{z}^{j,i}) + (\T^{q,W_{j,i}}(\bar{z}^{j,i}) - f(\bar{z}^{j,i})) + \calR^{q,W_{j,i}}(z^*_{j,i}) \\
        &\geq f(\bar{z}^{j,i}) + (\T^{q,W_{j,i}}(\bar{z}^{j,i}) - f(\bar{z}^{j,i})) - K \Delta^{q+1}.
    \end{aligned}
\end{equation}
In particular, for any $p \in \N$, we have
\begin{align*}
    \Lambda^p(x^{j,i},\Delta_j)
    &= \frac{f(x^{j,i}) - f(z^*_{j,i})}{\Delta_j^p} \\
    &\leq \frac{f(x^{j,i}) - f(\bar{z}^{j,i})}{\Delta_j^p} + \frac{f(\bar{z}^{j,i}) - \T^{q,W_{j,i}}(\bar{z}^{j,i})}{\Delta_j^p} + K \Delta_j^{q-p+1}.
\end{align*}
If $W_{j,i}$ is computed via \algApproxW{}, then by the stopping criterion in Step \ref{state:approx_W_stopping_criterion}, we obtain
\begin{align*}
    \Lambda^p(x^{j,i},\Delta_j)
    &\leq \frac{f(x^{j,i}) - f(\bar{z}^{j,i})}{\Delta_j^p} + \Delta_j^{q-p+\sigma} + K \Delta_j^{q-p+1}.
\end{align*}
If $f$ is bounded below, then the fraction $(f(x^{j,i}) - f(\bar{z}^{j,i})) / \Delta_j^p$ must become arbitrarily small (or negative) when increasing $i$ for any fixed $j$. In particular, for any vanishing sequence $(\tau_j)_j \subseteq \R^{>0}$ and for each $j \in \N$, we can choose $N_j$ such that
\begin{align} \label{eq:deriv_1}
    \Lambda^p(x^{j,N_j},\Delta_j) < \tau_j + \Delta_j^{q-p+\sigma} + K \Delta_j^{q-p+1}.
\end{align}
Denote $x^j := x^{j,N_j} = x^{j+1,0}$ for $j \in \N$. If $q \geq p$ then the right-hand side of \eqref{eq:deriv_1} vanishes for $j \rightarrow \infty$, so $\Lambda^p(x^j,\Delta_j) \rightarrow 0$. Since $\Lambda^p(x^j,\Delta_j) > \Lambda^1(x^j,\Delta_j)$ for $\Delta_j < 1$, it also holds $\Lambda^1(x^j,\Delta_j) \rightarrow 0$, which implies that all accumulation points of $(x^j)_j$ are critical points of $f$ by Subsec.\ \ref{subsec:prelim_nonsmooth_analysis} (cf.\ \eqref{eq:lambda_goldstein_estimate}). The resulting trust-region method is summarized in \algGlobal{}. (Note that the variable $x^j$ in Step \ref{state:global_method_change_j} is only used for ease of notation in our analysis.)
\begin{algorithm} 
    \caption{Globally convergent higher-order trust-region bundle method}
    \label{algo:global_method}
    \begin{algorithmic}[1] 
        \Require Oracle \ref{oracle:1}, initial point $x^0 = x^{1,0} \in \R^n$, vanishing sequences $(\Delta_j)_j$, $(\tau_j)_j \subseteq \R^{> 0}$, $p \in \N$, model order $q \in \N$, parameters $\sigma \in (0,1)$ and $c > 0$ (for \algApproxW{}).
        \For{$j = 1,2,\dots$}
            \For{$i = 0,1,\dots$}
                \State Compute $W_{j,i} \subseteq \Bcl_{\Delta_j}(x^{j,i})$ via \algApproxW{} (with $W^1 = \{ x^{j,i} \}$) and \label{state:global_method_approx_W}
                \Statex \hspace{\algorithmicindent}\hspace{\algorithmicindent}set $\bar{z}^{j,i} = \bar{z}^{q,W_{j,i}}(x^{j,i},\Delta_j)$.
                \If{$(f(x^{j,i}) - f(\bar{z}^{j,i})) / \Delta_j^p < \tau_j$} \label{state:global_method_decrease_condition}
                    \State Break $i$-loop.
                \Else
                    \State Set $x^{j,i+1} = \bar{z}^{j,i}$. \label{state:global_method_new_iterate}
                \EndIf
            \EndFor
            \State Set $x^{j+1,0} = x^{j,i}$ and $x^j = x^{j,i}$. \label{state:global_method_change_j}
        \EndFor
    \end{algorithmic}
\end{algorithm}
The derivation in this section directly leads to the following convergence result:
 
\begin{theorem} \label{thm:global_convergence}
    Let $q, p \in \N$, $q \geq p$, and $x^0 \in \R^n$. Assume that $f$ satisfies \ref{assum:B1} and that the sublevel set $\{ x \in \R^n : f(x) \leq f(x^0) \}$ is bounded. Then the sequence $(x^j)_j$ generated by \algGlobal{} satisfies $\Lambda^p(x^j,\Delta_j) \rightarrow 0$. Furthermore, $(x^j)_j$ has an accumulation point and all accumulation points are critical points of $f$.
\end{theorem}
\begin{proof}
    By construction, $(f(x^j))_j$ is a decreasing sequence, so $x^j$ lies in the bounded sublevel set $\{ x \in \R^n : f(x) \leq f(x^0) \}$ for all $j \in \N$. This implies that Lem.\ \ref{lem:q_order_error_estimate} can be applied as in inequality \eqref{eq:deriv_2}. Since $f$ is continuous, this sublevel set is compact, so that $(x^j)_j$ has an accumulation point. The rest of the proof follows via \eqref{eq:deriv_1} and Subsec.\ \ref{subsec:prelim_nonsmooth_analysis}, as discussed above.
\end{proof}

By Thm. \ref{thm:global_convergence}, if one is only concerned with global convergence, then it suffices to choose $q = p = 1$ in \algGlobal{}, in which case the model $\T^{q,W}$ is a standard cutting-plane model. However, if the property \eqref{eq:property_P} holds for one of the accumulation points, then choosing larger $p$ and $q$ immediately leads to the following corollary:
 
\begin{corollary} \label{cor:minimum_in_trust_region}
    Let $q, p \in \N$, $q \geq p$, and $x^0 \in \R^n$. Assume that $f$ satisfies \ref{assum:B1} and that the sublevel set $\{ x \in \R^n : f(x) \leq f(x^0) \}$ is bounded. Let $(x^j)_j$ be the sequence generated by \algGlobal{}. If \eqref{eq:property_P} holds for the chosen $p$ and an accumulation point $x^*$ of $(x^j)_j$, then $x^* \in \Bcl_{\Delta_j}(x^j)$ for infinitely many $j \in N$.
\end{corollary}
\begin{proof}
    Assume that $x^* \in \Bcl_{\Delta_j}(x^j)$ only holds finitely many times. Then there is some $N \in \N$ such that $x^* \notin \Bcl_{\Delta_j}(x^j)$ for all $j > N$. Since $x^*$ is assumed to be an accumulation point of $(x^j)_j$, there must be infinitely many $j > N$ with $x^j \in U$ for $U$ as in \eqref{eq:property_P}, which implies $\Lambda^p(x^j,\Delta_j) \geq C$ for such $j$ by \eqref{eq:property_P}. However, by Thm.\ \ref{thm:global_convergence}, it holds $\Lambda^p(x^j,\Delta_j) \rightarrow 0$, which is a contradiction.
\end{proof}

In words, Cor.\ \ref{cor:minimum_in_trust_region} shows that if \algGlobal{} converges to a point $x^*$ for which \eqref{eq:property_P} holds for the chosen $p$, and if the chosen model order $q$ is greater or equal to $p$, then $x^*$ is contained in infinitely many trust regions of vanishing size, which was the goal of this work. Unfortunately, it is not intuitively clear when \eqref{eq:property_P} holds for a \lc{2} function. As such, an analysis of \eqref{eq:property_P} is required to turn Cor.\ \ref{cor:minimum_in_trust_region} into a usable result, which is done in Sec.\ \ref{sec:property_P_for_max_type}. Further properties of \algGlobal{} and possible modifications will be discussed in Sec.\ \ref{sec:discussion_and_outlook}.

\section{Property (P) for finite max-type functions} \label{sec:property_P_for_max_type}

To motivate the strategy for this section, we first consider a necessary condition for \eqref{eq:property_P} to hold when $x^*$ is a local minimum. To this end, let $f : \R^n \rightarrow \R$ be a continuous function and assume that \eqref{eq:property_P} holds for a local minimum $x^*$ of $f$, $U \subseteq \R^n$, and $p \in \N$. Let $x \in U$. By considering the limit of $\Lambda^p(x,\Delta)$ for $\Delta$ approaching $\| x - x^* \|$ from below, it is easy to show that $\Lambda^p(x,\| x - x^* \|) \geq C$. Since $x^*$ is a local minimum, we can choose $U$ small enough to have $f(z^*(x,\| x - x^* \|)) = f(x^*)$. Then we see that
\begin{align*}
    f(x) - f(x^*)
    = f(x) - f(z^*(x,\| x - x^* \|))
    \geq C \| x - x^* \|^p
    \quad \forall x \in U,
\end{align*}
i.e., that $f$ has $p$-order growth around $x^*$ (cf.\ \eqref{eq:def_p_order_growth}). In other words, $p$-order growth is a necessary condition for property \eqref{eq:property_P} to hold (at local minima). Unfortunately, it is not sufficient, as the following example (from \cite{G2025}) shows:
 
\begin{example} \label{example:polynomial_growth_not_sufficient} 
    For $p \in \N$ consider the locally Lipschitz continuous function
    \begin{align*}
        f : \R \rightarrow \R, \quad x \mapsto
        \begin{cases}
            x^{p+1} \sin \left( \frac{1}{x} \right) + \frac{1}{p} |x|^p, & x \neq 0, \\
            0, & x = 0.
        \end{cases}
    \end{align*}
    Fig.\ \ref{fig:example_polynomial_growth_not_sufficient} shows the graph of $f$ for $p \in \{1,2,4\}$.
    \begin{figure}
        \centering
        \parbox[b]{0.32\textwidth}{
            \centering 
            \includegraphics[width=0.31\textwidth]{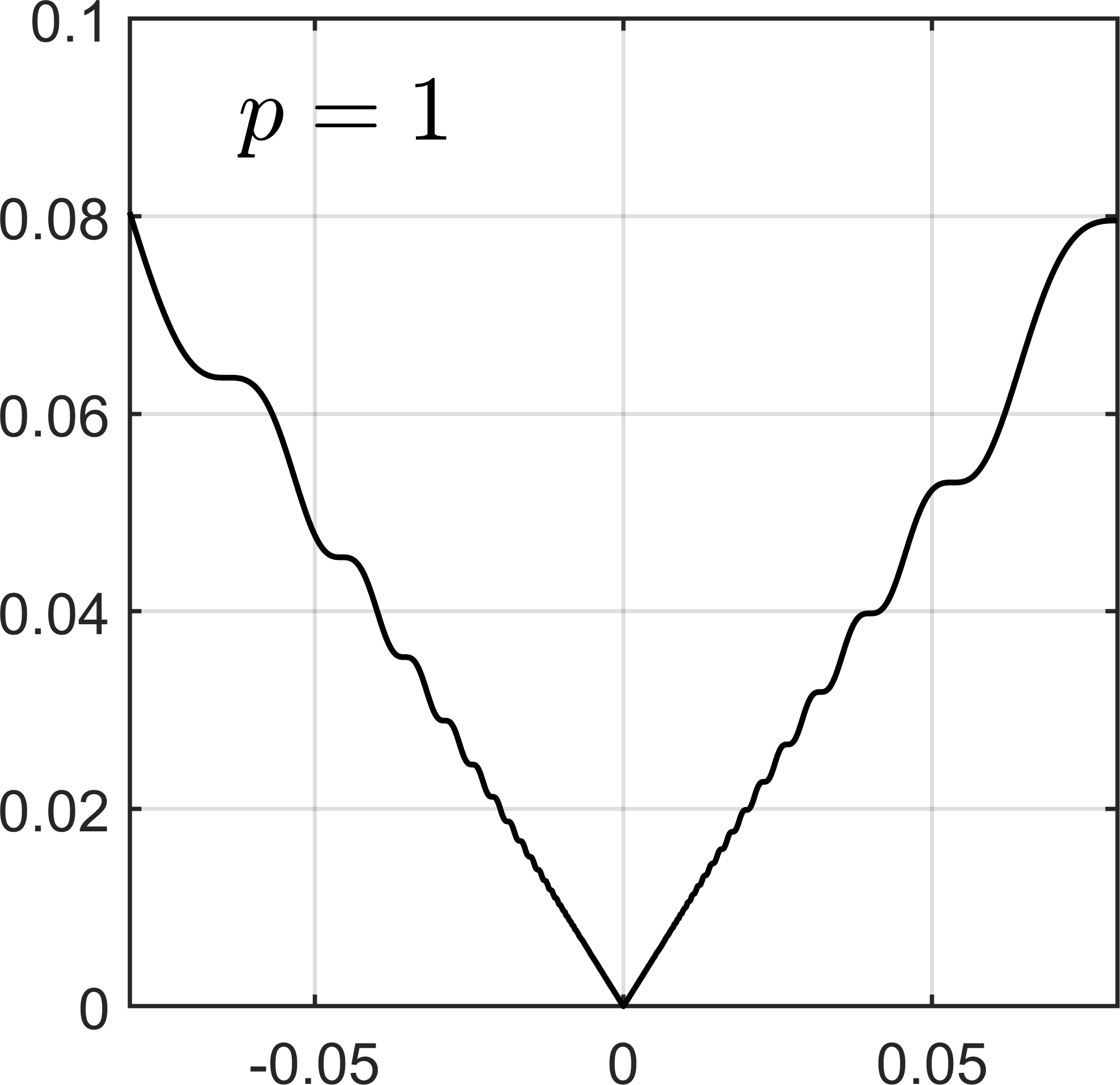}\\
        }
        \parbox[b]{0.32\textwidth}{
            \centering 
            \includegraphics[width=0.30\textwidth]{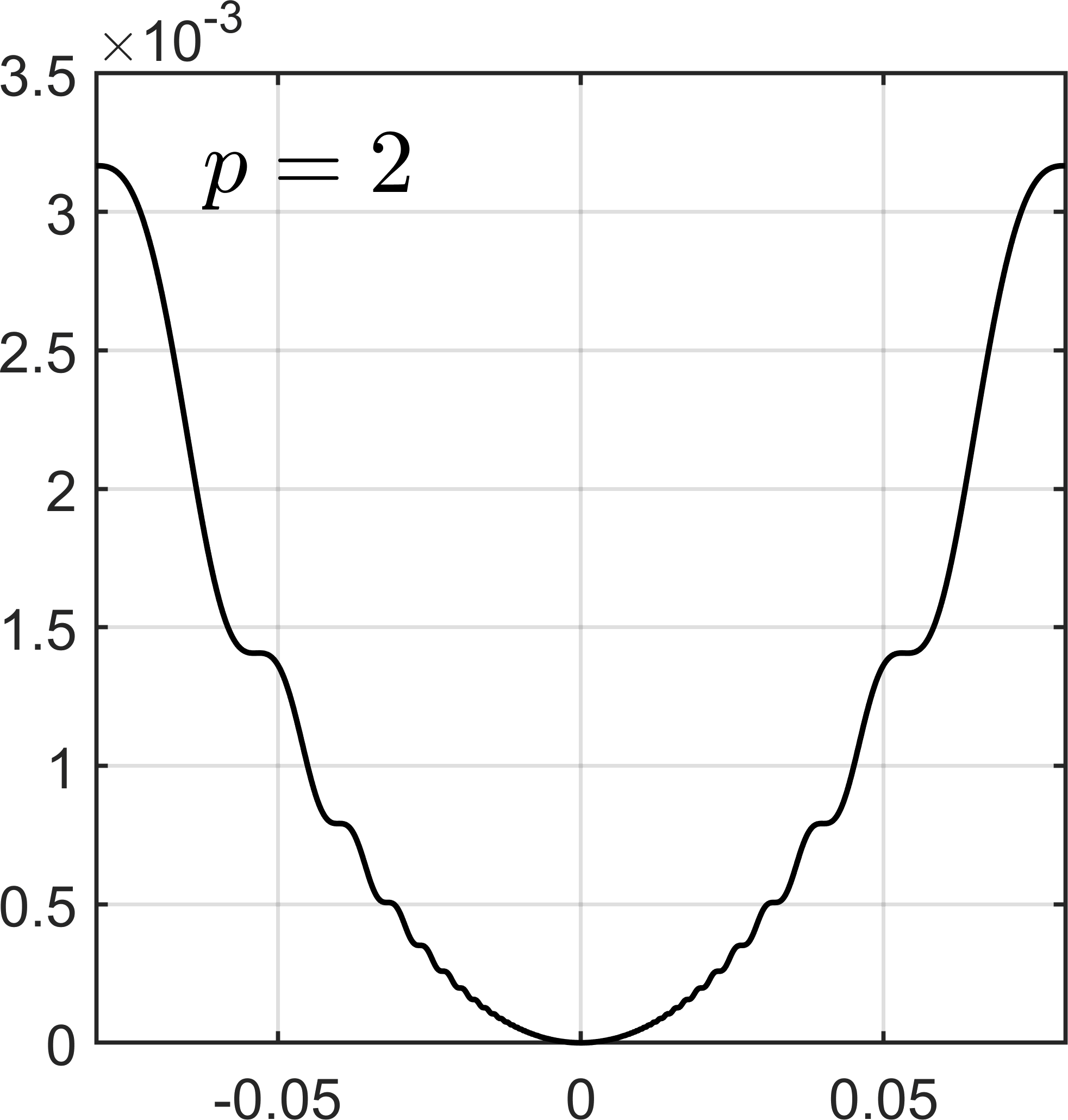}\\
        }
        \parbox[b]{0.32\textwidth}{
            \centering 
            \includegraphics[width=0.30\textwidth]{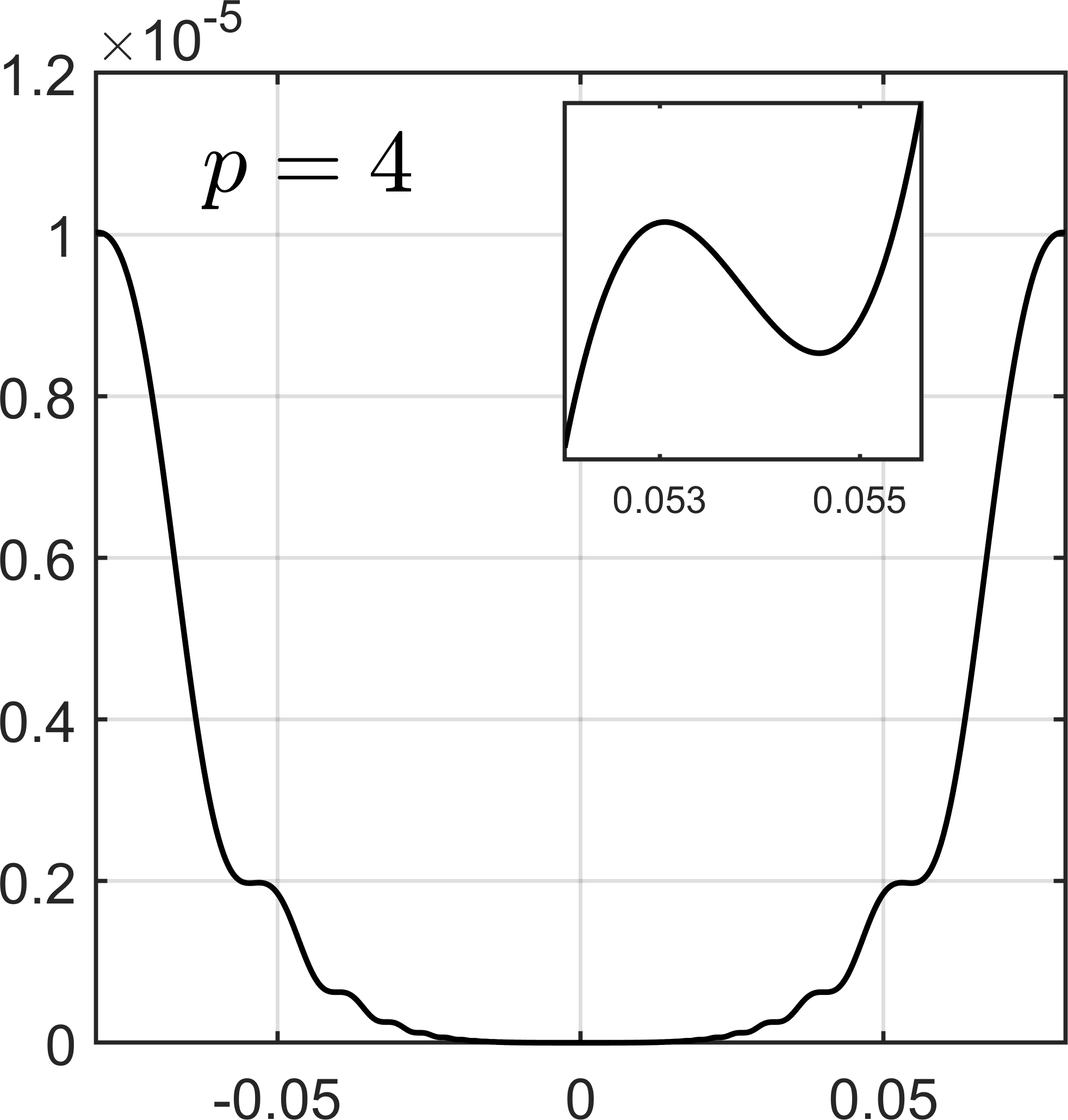}\\
        }
        \caption{The graph of $f$ in Ex.\ \ref{example:polynomial_growth_not_sufficient} for $p \in \{1,2,4\}$. For $p = 4$, the small additional plot shows a zoom for better visualization of the local minima (and maxima).}
        \label{fig:example_polynomial_growth_not_sufficient}
    \end{figure}
    It is possible to show that for $p \geq 2$, $f$ is $\C^1$, and for $p \geq 4$, $f$ is $\C^2$. Since the sine is bounded and $x^{p+1}$ vanishes faster than $|x|^p$ for $x \rightarrow 0$, it is easy to see that $x^* = 0$ is a local minimum around which $f$ grows with order $p$. However, it can be shown that there are infinitely many other local minima arbitrarily close to $x^*$. Since for each local minimum $\hat{x}$, $\Delta$ can be chosen small enough to have $\Lambda^p(\hat{x},\Delta) = 0$, this shows that \eqref{eq:property_P} cannot hold for $x^*$.
\end{example}
 
By the previous example, for general locally Lipschitz continuous functions, $p$-order growth alone is not enough to infer \eqref{eq:property_P}. While there is still hope that the situation is better for \lc{2} functions, we will only be able to infer \eqref{eq:property_P} by imposing additional assumptions. These assumptions include finiteness of $S$ in \ref{assum:B1}, which means that the functions are finite max-type functions. To not overload the current work, in this section, we only consider the growth orders $p = 1$ (i.e., sharp growth) and $p = 2$ (i.e., quadratic growth). The generalization to general \lc{2} functions (with infinite $S$) and arbitrary orders of growth will be discussed in the outlook in Sec.\ \ref{sec:discussion_and_outlook}.

As a start, for both cases, we require the technical result that the difference $z^*(x,\Delta) - x$ can be written as a scaled convex combination of gradients of selection functions active at $z^*(x,\Delta)$, which can be derived using the epigraph reformulation of the optimization problem in the definition of $z^*(x,\Delta)$ (cf.\ \eqref{eq:def_z_star_Lambda}). For the sake of brevity, we write $z^* = z^*(x,\Delta)$ whenever the context allows, and for ease of notation, we make the following definition:
 
\begin{definition} \label{def:simplex_lagrange}
    Let $S$ be a finite set. Denote the standard simplex by
    \begin{align*}
        \Omega_S := \left\{ \lambda \in \R^{|S|} : \lambda_s \geq 0 \ \forall s \in S, \ \sum_{s \in S} \lambda_s = 1 \right\}.
    \end{align*}
    If $f $ satisfies \ref{assum:B1} with finite index set $S$ and $x \in \R^n$ is a critical point of $f$, then we call a vector $\lambda \in \Omega_{A(x)}$ with $\sum_{s \in A(x)} \lambda_s \nabla f_s(x) = 0$ a \emph{Lagrange multiplier} of $f$ at $x$.
\end{definition}
 
Using this definition, the following formula for $z^*(x,\Delta) - x$ can be derived:
 
\begin{lemma} \label{lem:z_star_formula}
    Let $q \in \N$ and assume that $f$ satisfies \ref{assum:B1} with finite $S$. Let $x \in \R^n$ and $\Delta > 0$. If $z^* = z^*(x,\Delta)$ is not a critical point of $f$, then there is some $\lambda \in \Omega_{A(z^*)}$ such that
    \begin{align*}
         z^* - x = -\Delta \frac{\sum_{s \in A(z^*)} \lambda_s \nabla f_s(z^*)}{\| \sum_{s \in A(z^*)} \lambda_s \nabla f_s(z^*) \|}.
    \end{align*}
\end{lemma}
\begin{proof}
    By definition, $z^*$ is a solution of
    \begin{equation*}
        \begin{aligned}
            \min_{z \in \R^n, \theta \in \R} \ & \theta \\
            \text{s.t.} \ & f_s(z) \leq \theta \quad \forall s \in S,\\
            & \tfrac{1}{2} \| z - x \|^2 \leq \tfrac{1}{2} \Delta^2.
        \end{aligned}
    \end{equation*}
    It is easy to see that the constraints satisfy the MFCQ (see, e.g., \cite{NW2006}), so the optimality conditions imply that there are $\lambda_s \geq 0$, $s \in S$, and $\lambda^b \geq 0$ such that
    \begin{align} \label{eq:proof_lem_z_star_formula_1}
        0 = 
        \begin{pmatrix}
            0 \\
            1
        \end{pmatrix}
        + \sum_{s \in S} \lambda_s
        \begin{pmatrix}
            \nabla f_s(z^*) \\
            -1
        \end{pmatrix}
        + \lambda^b
        \begin{pmatrix}
            z^* - x \\
            0
        \end{pmatrix},
    \end{align}
    with $\lambda^b = 0$ if $\| z^* - x \| < \Delta$ and $\lambda_s = 0$ if $s \notin A(z^*)$. In particular, $S$ in \eqref{eq:proof_lem_z_star_formula_1} can be replaced by $A(z^*)$. The second line then shows that $(\lambda_s)_{s \in A(z^*)} \in \Omega_{A(z^*)}$. Since $z^*$ is assumed to not be a critical point, $\lambda^b$ must be positive and $\| z^* - x \| = \Delta$. By \eqref{eq:proof_lem_z_star_formula_1}, this implies
    \begin{align*}
        \lambda^b = \frac{\| \sum_{s \in A(z^*)} \lambda_s \nabla f_s(z^*) \|}{\Delta},
    \end{align*}
    so
    \begin{align*}
        z^* - x 
        = -\frac{1}{\lambda^b} \sum_{s \in A(z^*)} \lambda_s \nabla f_s(z^*)
        = - \Delta \frac{\sum_{s \in A(z^*)} \lambda_s \nabla f_s(z^*)}{\| \sum_{s \in A(z^*)} \lambda_s \nabla f_s(z^*) \|},
    \end{align*}
    completing the proof.
\end{proof}

\subsection{Sharp growth}

    In this subsection, we consider functions with a sharp local minimum $x^* \in \R^n$, which means that \eqref{eq:def_p_order_growth} holds for $p = 1$. The following lemma shows that for functions satisfying \ref{assum:B1}, sharpness implies that the norm of all subgradients close to $x^*$ is bounded away from zero. (In particular, the situation for $p = 1$ in Ex.\ \ref{example:polynomial_growth_not_sufficient} cannot occur for such functions.)
     
    \begin{lemma} \label{lem:sharp_subgradient_bound}
        Let $q \in \N$ and assume that $f$ satisfies \ref{assum:B1}. If $f$ grows with order $p = 1$ and constant $\beta > 0$ around $x^* \in \R^n$, then there is an open neighborhood $U \subseteq \R^n$ of $x^*$ such that
        \begin{align*}
            \| \xi \| > \beta / 2 \quad \forall x \in U \setminus \{ x^* \}, \xi \in \partial f(x).
        \end{align*}
        In particular, $x^*$ is the only critical point of $f$ in $U$.
    \end{lemma}
    \begin{proof}
        Assume that this does not hold. Then there are sequences $(x^j)_j \subseteq \R^n$ and $(\xi^j)_j \subseteq \R^n$ with $x^j \rightarrow x^*$, $x^j \neq x^*$, $\xi^j \in \partial f(x^j)$, and $\| \xi^j \| \leq \beta/2$ for all $j \in \N$. For $j \in \N$ let $d^j = (x^j - x^*)/\| x^j - x^* \|$. Then $\| d^j \| = 1$ for all $j \in \N$, so by compactness, we can assume w.l.o.g.\ that $d^j \rightarrow d \in \R^n$ with $\| d \| = 1$. Denote $f'(x^*,d) := \lim_{t \searrow 0} (f(x^* + td) - f(x^*))/t$. Since $f$ is semismooth as a \lc{1} function (cf.\ \cite{S1981}, Thm.\ 3.9), we obtain 
        \begin{align*}
            f'(x^*,d) 
            = \lim_{j \rightarrow \infty} (\xi^j)^\top d^j 
            \leq \limsup_{j \rightarrow \infty} \| \xi^j \| \| d^j \|
            \leq \beta/2.
        \end{align*}
        However, the sharp growth of $f$ around $x^*$ implies
        \begin{align*}
            f'(x^*,d)
            = \lim_{t \searrow 0} \frac{f(x^* + td) - f(x^*)}{t}
            \geq \lim_{t \searrow 0} \frac{\beta t \| d \|}{t}
            = \beta,
        \end{align*}
        which is a contradiction.
    \end{proof}

    Lem.\ \ref{lem:sharp_subgradient_bound} shows that for $x$ close to $x^*$ and $x^* \notin \Bcl_\Delta(x)$, we can always apply Lem.\ \ref{lem:z_star_formula} to write $z^* - x$ as a scaled convex combination of gradients of active selection functions. In the following, we show that this can be used to obtain an estimate for $\Lambda^1(x,\Delta)$ in terms of the norm of certain subgradients close to $x^*$. Since these subgradients are bounded away from zero, this implies \eqref{eq:property_P}. More precisely, the following theorem holds:
     
    \begin{theorem} \label{thm:sharp_decrease_property}
        Let $q \in \N$ and assume that $f$ satisfies \ref{assum:B1} with finite $S$. If $f$ grows with order $p = 1$ around $x^* \in \R^n$, then \eqref{eq:property_P} holds (for $p = 1$).
    \end{theorem}
    \begin{proof}
        \textbf{Part 1:} Assume that this lemma does not hold. Then for all open neighborhoods $U \subseteq \R^n$ of $x^*$ and all $C > 0$, there are $x \in U$ and $\Delta > 0$ with $x^* \notin \Bcl_\Delta(x)$ such that $\Lambda^1(x,\Delta) < C$. In particular, there are sequences $(x^j)_j \subseteq \R^n$ and $(\Delta_j)_j \subseteq \R^{>0}$ with $x^j \rightarrow x^*$, $\Delta_j \rightarrow 0$, $x^* \notin \Bcl_{\Delta_j}(x^j)$ for all $j \in \N$ and
        \begin{align} \label{eq:proof_lem_sharp_decrease_property_1}
            0 
            = \lim_{j \rightarrow \infty} \Lambda^1(x^j,\Delta_j)
            = \lim_{j \rightarrow \infty} \frac{f(x^j) - f(z^*(x^j,\Delta_j))}{\Delta_j}.
        \end{align}
        For ease of notation, denote $z^j := z^*(x^j,\Delta_j)$. Since $S$ is finite, we can assume w.l.o.g.\ that $A(z^j) = S'$ is constant for all $j \in \N$. \\
        \textbf{Part 2:} By Lem.\ \ref{lem:sharp_subgradient_bound} there is an open neighborhood $U \subseteq \R^n$ of $x^*$ such that $x^*$ is the only critical point of $f$ in $U$. In particular, we can assume w.l.o.g.\ that $z^j$ is not a critical point for any $j \in \N$. By Lem.\ \ref{lem:z_star_formula}, for each $j \in \N$, this implies $\| z^j - x^j \| = \Delta_j$ and that there is some $\lambda^j \in \Omega_{S'}$ such that
        \begin{align} \label{eq:proof_lem_sharp_decrease_property_2}
            z^j - x^j = -\Delta_j \frac{\sum_{s \in S'} \lambda_s^j \nabla f_s(z^j)}{\| \sum_{s \in S'} \lambda_s^j \nabla f_s(z^j) \|}.
        \end{align}
        \textbf{Part 3:} Denote $\psi_j := \sum_{s \in S'} \lambda_s^j f_s$. Since $S' = A(z^j)$ and $\lambda^j \in \Omega_{S'}$ for all $j \in \N$, it holds
        \begin{align*}
            f(x^j) \geq \psi_j(x^j)
            \quad \text{and} \quad
            f(z^j) = \psi_j(z^j)
            \quad \forall j \in \N.
        \end{align*}
        First-order Taylor expansion of $\psi_j$ at $z^j$ with remainder $R_j$ (note that $\psi_j$, and therefore the remainder, depends on $j$ due to $\lambda^j$) yields
        \begin{align*}
            &f(x^j) - f(z^j)
            \geq \psi_j(x^j) -\psi_j(z^j) 
            = \nabla \psi_j(z^j)^\top (x^j - z^j)
            + R_j \\
            &\stackrel{\eqref{eq:proof_lem_sharp_decrease_property_2}}{=} \Delta_j \left\| \nabla \psi_j(z^j) \right\|
            + R_j,
        \end{align*}
        so
        \begin{align} \label{eq:proof_lem_sharp_decrease_property_3}
            \Lambda^1(x^j,\Delta_j)
            = \frac{f(x^j) - f(z^j)}{\Delta_j}
            \geq \left\| \sum_{s \in S'} \lambda_s^j \nabla f_s(z^j) \right\| + \frac{R_j}{\Delta_j}
        \end{align}
        for all $j \in \N$. By the remainder formula \eqref{eq:Taylor_remainder} for $R_j$, since all $f_s$, $s \in S$, are $\C^{2}$ and since a convex combination of these selection functions was expanded, the second summand on the right-hand side in the above inequality vanishes. Combined with \eqref{eq:proof_lem_sharp_decrease_property_1}, this implies that the first summand vanishes as well. \\
        \textbf{Part 4:} Since $S' = A(z^j)$ and $\lambda^j \in \Omega_{S'}$ it holds $\sum_{s \in S'} \lambda_s^j \nabla f_s(z^j) \in \partial f(z^j)$ for all $j \in \N$. Since $z^j \neq x^*$ and $z^j \rightarrow x^*$ by construction, Lem.\ \ref{lem:sharp_subgradient_bound} shows that the norm of $\sum_{s \in S'} \lambda_s^j \nabla f_s(z^j)$ must be bounded away from zero, which is a contradiction to Part 3.
    \end{proof}

\subsection{Quadratic growth}

    In this subsection, we consider functions with a quadratic local minimum $x^* \in \R^n$, which means that \eqref{eq:def_p_order_growth} holds for $p = 2$. For the proof of Thm.\ \ref{thm:sharp_decrease_property} in the sharp case, the idea was to derive the lower estimate \eqref{eq:proof_lem_sharp_decrease_property_3} for $\Lambda^1(x,\Delta)$ in terms of the derivatives of the selection functions $f_s$. Since the subgradients were bounded away from zero close to $x^*$, this only required a first-order Taylor expansion and, in particular, only first-order derivatives. In contrast to this, in the quadratic case, subgradients may be arbitrarily small near $x^*$. This means is that an estimate for $\Lambda^2(x,\Delta)$ requires second-order Taylor expansion, involving second-order derivatives of the selection functions. In particular, we require that the second-order term in the Taylor expansion is bounded away from zero. Unfortunately, we are only able to show this under the following, relatively strong assumptions on the representation in \ref{assum:B1}:
     
    \begin{assum} \label{assum:B2}
        Assume that $f$ satisfies \ref{assum:B1} for $q \geq 2$ and that it grows with order $p = 2$ around $x^* \in \R^n$. Furthermore, assume that for the representation of $f$ in \ref{assum:B1}, it holds
        \begin{enumerate}[itemindent=25pt,label=(B2.\arabic*)]
            \item \label{enum:B2_1} $S = \{ 1, \dots, |S| \}$ is finite,
            \item \label{enum:B2_2} $A(x^*) = S$ and there is a Lagrange multiplier $\lambda^* \in \Omega_S$ at $x^*$ (cf.\ Def.\ \ref{def:simplex_lagrange}) with $\lambda^*_s > 0$ for all $s \in S$,
            \item \label{enum:B2_3} all $\nabla f_s(x^*)$, $s \in S$, are affinely independent. 
        \end{enumerate}
    \end{assum}
     
    The affine independence in \ref{enum:B2_3} implies that the Lagrange multiplier $\lambda^*$ in \ref{enum:B2_2} is unique. Furthermore, let $\calM := \{ x \in \R^n : A(x) = S \}$ be the set of points at which all selection functions are active. Then $\calM = \Phi^{-1}(0)$ for
    \begin{align*}
        \Phi : \R^n \rightarrow \R^{|S| - 1}, \quad x \mapsto 
        \left( f_2(x) - f_1(x), \dots, f_{|S|}(x) - f_1(x) \right)^\top.
    \end{align*}
    By \ref{enum:B2_3} and continuity of $\nabla f_s$, $s \in S$, there is an open neighborhood $U \subseteq \R^n$ of $x^*$ such that the Jacobian $D \Phi(x)$ has full rank for all $x \in U$. This implies that $\calM \cap U$ is an $(n - |S| + 1)$-dimensional manifold with tangent space $\ker(D \Phi(x))$ at $x \in \calM \cap U$ (cf.\ \cite{L2012}, Cor.\ 5.14 and Prop.\ 5.38). In particular, for $x \in \calM \cap U$ and any open interval $I \subseteq \R$ with $0 \in I$,
    \begin{equation} \label{eq:tangent_space_via_curves}
        \begin{aligned}
            \ker(D \Phi(x)) = \{ &v \in \R^n : \exists \ \C^1 \text{ curve } \varphi : I \rightarrow \calM \cap U \ \\
            &\text{with} \ \varphi(0) = x, \ \varphi'(0) = v \}
        \end{aligned}
    \end{equation}
    (cf.\ \cite{L2012}, p.\ 72). Assumptions similar to \ref{assum:B2} have appeared in other works, see, e.g., Def.\ 5.2 in \cite{LW2019} (``strong second-order conditions'') or Def.\ 3 in \cite{HL2023} (``strong $\C^2$ max function''). More generally, the manifold structure of $\calM$ is related to partial smoothness \cite{L2002} and $\mathcal{V} \mathcal{U}$-theory \cite{MS2005,LS2020}.

    Note that when restricted to $\calM$, $f$ behaves like any of its smooth selection functions $f_s$. This allows us to simplify the analysis of the behavior of $f$ around $x^*$ by splitting it up into the smooth behavior in $\calM$ and the nonsmooth behavior orthogonal to $\calM$ (as is done in $\mathcal{V} \mathcal{U}$-theory). Concerning the latter, the following lemma shows that $f$ grows sharply when moving orthogonally away from $\calM$. To prove it, it is worth using the concepts of \emph{affine hull} and \emph{relative boundary} of $\partial f(x^*)$ (see, e.g, \cite{B1983}, p.\ 19), which we denote by $\aff(\partial f(x^*))$ and $\rb(\partial f(x^*))$, respectively. By \ref{enum:B2_1} and \cite{B1983}, Exercise 3.1, $\rb(\partial f(x^*))$ is the set of convex combinations of gradients of selection functions where at least one coefficient is zero. (In particular, $\rb(\partial f(x^*))$ is compact.)
     
    \begin{lemma} \label{lem:quad_orth_lin_growth}
        Assume that $f$ satisfies \ref{assum:B2}.
        \begin{enumerate}[label=(\alph*)]
            \item There is a constant $\delta > 0$, an open neighborhood $U \subseteq \R^n$ of $x^*$ and some $\bar{t} > 0$ such that
            \begin{align*}
                \frac{f(y + tw) - f(y)}{t} > \delta \quad \forall y \in U \cap \calM, t \in (0,\bar{t}], w \in \ker(D \Phi(y))^\bot, \| w \| = 1.
            \end{align*}
            \item There is a constant $\delta > 0$ and an open neighborhood $U \subseteq \R^n$ of $x^*$ such that for every $x \in U \setminus \calM$ and every $\pr(x) \in \argmin_{y \in \calM} \| y - x \|$, it holds
            \begin{align*}
                \frac{f(x) - f(\pr(x))}{\| x - \pr(x) \|} > \delta.
            \end{align*}
        \end{enumerate}
    \end{lemma}
    \begin{proof}
        \textbf{(a) Part 1:} Let $w \in \ker(D \Phi(x^*))^\bot$ with $\| w \| = 1$. Since we have $\ker(D \Phi(x^*))^\bot = \vspan(\{ \nabla f_s(x^*) - \nabla f_1(x^*) : s \in S, s \neq 1 \})$, $w$ can be written as a linear combination of $\nabla f_s(x^*)$, $s \in S$, where the sum of coefficients is zero. As $\partial f(x^*) = \conv(\{ \nabla f_s(x^*) : s \in S \})$ and $0 \in \partial f(x^*) \subseteq \aff(\partial f(x^*))$, this means that $\gamma w \in \aff(\partial f(x^*))$ for all $\gamma \in \R$. Let 
        \begin{align*}
            \delta' := \min_{\xi \in \rb(\partial f(x^*))} \|\xi \|.
        \end{align*}
        Due to uniqueness of $\lambda^*$ in \ref{enum:B2_2}, we have $0 \notin \rb(\partial f(x^*))$, and as $\rb(\partial f(x^*))$ is closed, we have $\delta' > 0$. Since $\partial f(x^*)$ is bounded and $\| w \| = 1$, there is some $\gamma \geq \delta'$ such that $\gamma w \in \rb(\partial f(x^*))$. In particular, it holds
        \begin{align} \label{eq:proof_lem_quad_orth_lin_growth_1}
            \max_{\xi \in \partial f(x^*)} \xi^\top w \geq \gamma w^\top w = \gamma \| w \|^2 = \gamma \geq \delta'.
        \end{align}
        \textbf{Part 2:} Assume that (a) does not hold. Then there are sequences $(y^j)_j \subseteq \calM$, $(t_j)_j \subseteq \R^{>0}$, $(w^j)_j \subseteq \R^n$, and $(\delta_j)_j \subseteq \R^{>0}$ with $y^j \rightarrow x^*$, $t_j \rightarrow 0$, $\delta_j \rightarrow 0$, and $w^j \in \ker(D\Phi(y^j))^\bot$, $\| w^j \| = 1$ for all $j \in \N$, such that
        \begin{align*}
            \frac{f(y^j + t_j w^j) - f(y^j)}{t_j} \leq \delta_j \quad \forall j \in \N.
        \end{align*}
        Note that $A(y^j) = S$ implies $f(y^j) = f_s(y^j)$ for all $j \in \N, s \in S$, so
        \begin{align*}
            \frac{f_s(y^j + t_j w^j) - f_s(y^j)}{t_j} 
            \leq \frac{f(y^j + t_j w^j) - f(y^j)}{t_j}
            \leq \delta_j \quad \forall j \in \N, s \in S.
        \end{align*}
        As all $f_s$ are $\C^2$ and $y^j \rightarrow x^*$, first-order Taylor expansion of $f_s$ at $y^j$ yields
        \begin{align} \label{eq:proof_lem_quad_orth_lin_growth_2}
            \nabla f_s(y^j)^\top w^j + \frac{O(\| t_j w^j \|^2)}{t_j} \leq \delta_j \quad \forall j \in \N, s \in S.
        \end{align}
        Since $\| w^j \| = 1$ for all $j \in \N$, we can assume w.l.o.g.\ that $w^j \rightarrow w$ with $\| w \| = 1$. As all $\nabla f_s$, $s \in S$, are continuous and affinely independent at $x^*$ and $w^j \in \ker(D\Phi(y^j))^\bot$ for all $j \in \N$, we have $w \in \ker(D \Phi(x^*))^\bot$. Letting $j \rightarrow \infty$ in \eqref{eq:proof_lem_quad_orth_lin_growth_2} yields $\nabla f_s(x^*)^\top w \leq 0$ for all $s \in S$. This is a contradiction to \eqref{eq:proof_lem_quad_orth_lin_growth_1}, since $\xi \in \partial f(x^*)$ is a convex combination of $\nabla f_s(x^*)$, $s \in S$. \\
        \textbf{(b) Part 1:} Let $\alpha > 0$ and $x \in \Bcl_{\alpha}(x^*)$. Since $\calM \cap \Bcl_{3 \alpha}(x^*)$ is compact the optimal value of 
        \begin{align} \label{eq:proof_lem_quad_orth_lin_growth_3}
            \min_{y \in \calM \cap \Bcl_{3 \alpha}(x^*)} \| y - x \|
        \end{align}
        exists, and since $x^* \in \calM$, it is less or equal $\| x^* - x  \| \leq \alpha$. In particular, all minima of \eqref{eq:proof_lem_quad_orth_lin_growth_3} lie in $\Bcl_{\alpha}(x) \subseteq \Bcl_{2 \alpha}(x^*)$. This implies that the constraint $y \in \Bcl_{3 \alpha}(x^*)$ in \eqref{eq:proof_lem_quad_orth_lin_growth_3} is superfluous, so $\min_{y \in \calM} \| y - x \|$ has a solution for all $x \in \Bcl_\alpha(x^*)$. \\
        \textbf{Part 2:} Let $x \in \Bcl_\alpha(x^*)$ and $\pr(x) \in \argmin_{y \in \calM} \| y - x \|$. By definition, $\pr(x)$ is a solution of
        \begin{equation} \label{eq:proof_lem_quad_orth_lin_growth_4}
            \begin{aligned}
                \min_{y \in \R^n} \ & \| y - x \|^2 \\
                \text{s.t.} \ & f_s(y) - f_1(y) = 0 \quad \forall s \in S, s \neq 1.
            \end{aligned}
        \end{equation}
        By Part 1, $\pr(x)$ must lie in $\Bcl_{\alpha}(x) \cap \calM \subseteq \Bcl_{2\alpha}(x^*) \cap \calM$.
        By \ref{enum:B2_3} and continuity of $\nabla f_s$, $s \in S$, we can choose $\alpha$ small enough so that all $\nabla f_s(y)$, $s\in S$, are affinely independent for all $y \in \Bcl_{2\alpha}(x^*)$, which implies that the LICQ holds in \eqref{eq:proof_lem_quad_orth_lin_growth_4}. The first-order optimality conditions yield the existence of some $\lambda \in \R^{|S|}$ with
        \begin{align*}
            0 = 2(\pr(x) - x) + \sum_{s \in S, s \neq 1} \lambda_s (\nabla f_s(\pr(x)) - \nabla f_1(\pr(x))),
        \end{align*}
        which is equivalent to 
        \begin{align*}
            x = \pr(x) + \frac{1}{2} \sum_{s \in S, s \neq 1} \lambda_s (\nabla f_s(\pr(x)) - \nabla f_1(\pr(x))).
        \end{align*}
        Since $\ker(D \Phi(\pr(x)))^\bot = \vspan(\{ \nabla f_s(\pr(x)) - \nabla f_1(\pr(x)) : s \in S, s \neq 1 \})$, this shows that there is some $w' \in \ker(D \Phi(\pr(x)))^\bot$ such that $x = \pr(x) + w'$. Let $t := \| w'\| = \|{\pr}(x) - x \| \leq \alpha$ and $w := w'/\| w'\|$.\\
        \textbf{Part 3:} Let $\delta$ as in (a). Then we can choose $\alpha$ small enough so that
        \begin{align*}
            \frac{f(x) - f(\pr(x))}{\| x - \pr(x) \|}
            = \frac{f(\pr(x) + t w) - f(\pr(x))}{t}
            \stackrel{(a)}{>} \delta
        \end{align*}
        for all $x \in \Bcl_\alpha(x^*)$, which completes the proof (with $U$ being any open neighborhood of $x^*$ in $\Bcl_\alpha(x^*)$).
    \end{proof} 

    Having analyzed the growth of $f$ orthogonal to $\calM$, it remains to analyze the growth along $\calM$. To this end, note that due to the max-type structure of $f$, for any convex combination of selection functions $\sum_{s \in S} \lambda_s f_s$, we have $\sum_{s \in S} \lambda_s f_s(x) \leq f(x)$ for all $x \in \R^n$ and $\sum_{s \in S} \lambda_s f_s(x) = f(x)$ for all $x \in \calM$. In words, any convex combination of selection functions is a lower bounding function for $f$ that is identical to $f$ in $\calM$. Furthermore, for the Lagrange multiplier $\lambda^*$ from \ref{enum:B2_2}, $x^*$ is a critical point of $\sum_{s \in S} \lambda^*_s f_s$. Due to the quadratic growth of $f$, this allows us to show that the Hessian matrix of $\sum_{s \in S} \lambda^*_s f_s$ is positive definite along $\calM$: 
     
    \begin{lemma} \label{lem:quad_pos_definite}
        Assume that $f$ satisfies \ref{assum:B2}. Then
        \begin{align*}
            \frac{1}{2} v^\top \left( \sum_{s \in S} \lambda^*_s \nabla^2 f_{s}(x^*) \right) v \geq \beta
        \end{align*}
        for all $v \in \ker(D \Phi(x^*))$ with $\| v \| = 1$.
    \end{lemma}
    \begin{proof}
        Let $U$ be small enough so that $D \Phi(x)$ has full rank for all $x \in U$. Let $v \in \ker(D \Phi(x^*))$ with $\| v \| = 1$. Let $\varphi : I \rightarrow \calM \cap U$ be a curve corresponding to $v$ as in \eqref{eq:tangent_space_via_curves}, i.e., $\varphi(0) = x^*$ and $\varphi'(0) = v$. Let $\lambda^* \in \Omega_S$ as in \ref{enum:B2_2}. Since $A(x^*) = S$ and all $f_s$, $s \in S$, are $\C^3$, second-order Taylor expansion of $\psi_* := \sum_{s \in S} \lambda^*_s f_s$ at $x^*$ yields
        \begin{align*}
            &f(\varphi(t)) - f(x^*)
            = \psi_*(\varphi(t)) - \psi_*(x^*) \\
            &= \frac{1}{2} (\varphi(t) - x^*)^\top \nabla ^2 \psi_*(x^*) (\varphi(t) - x^*) + O(\| \varphi(t) - x^* \|^3)
        \end{align*}
        for all $t \in I$. Quadratic growth of $f$ around $x^*$ leads to
        \begin{align*}
            \beta
            &\leq \frac{f(\varphi(t)) - f(x^*)}{\| \varphi(t) - x^* \|^2} \\
            &= \frac{1}{2} \left( \frac{\varphi(t) - x^*}{\| \varphi(t) - x^* \|} \right)^\top \left( \sum_{s \in S} \lambda^*_s \nabla^2 f_{s}(x^*) \right) \frac{\varphi(t) - x^*}{\| \varphi(t) - x^* \|} + \frac{O(\| \varphi(t) - x^* \|^3)}{\| \varphi(t) - x^* \|^2}
        \end{align*}
        for all $t \in I$. Letting $t \rightarrow 0$ completes the proof, since
        \begin{align*}
            \frac{\varphi(t) - x^*}{\| \varphi(t) - x^* \|}
            = \left\| \frac{\varphi(t) - \varphi(0)}{t} \right\|^{-1} \frac{\varphi(t) - \varphi(0)}{t} 
            \xrightarrow[]{t \rightarrow 0} \| \varphi'(0) \|^{-1} \varphi'(0) 
            = v.
        \end{align*}
    \end{proof}
    
    The following lemma is a technical result on sequences of differences of two points in $\calM$, which will be required for applying Lem.\ \ref{lem:quad_pos_definite} in our context:
     
    \begin{lemma} \label{lem:quad_tangent_sequence}
        Assume that $f$ satisfies \ref{assum:B2}. Let $(x^j)_j$ and $(y^j)_j$ be sequences in $\calM$ with limit $x^*$. Then every accumulation point of $((x^j - y^j)/\| x^j - y^j \|)_j$ lies in $\ker(D\Phi(x^*))$.
    \end{lemma}
    \begin{proof}
        For $s \in S$, $s \neq 1$, denote $h(x) := f_s(x) - f_1(x)$. By the mean-value theorem, for every $j \in \N$, there is some $a^j \in \conv(\{ x^j, y^j \})$ with
        \begin{align*}
            0
            = \frac{h(x^j) - h(y^j)}{\| x^j - y^j \|}
            = \frac{\nabla h(a^j)^\top (x^j - y^j)}{\| x^j - y^j \|}
            = (\nabla f_s(a^j) - \nabla f_1(a^j))^\top \frac{x^j - y^j}{\| x^j - y^j \|}.
        \end{align*}
        Since $x^j \rightarrow x^*$ and $y^j \rightarrow x^*$ it also holds $a^j \rightarrow x^*$, so letting $j \rightarrow \infty$ completes the proof.
    \end{proof}

    Finally, to be able to use the formula for $z^*(x,\Delta) - x$ from Lem.\ \ref{lem:z_star_formula}, we have to show that there are no other critical points of $f$ close to $x^*$, which we do in the following lemma:
     
    \begin{lemma} \label{lem:quad_isolated_critical_point}
        Assume that $f$ satisfies \ref{assum:B2}. Then there is an open neighborhood $U \subseteq \R^n$ of $x^*$ such that $x^*$ is the only critical point of $f$ in $U$. 
    \end{lemma}
    \begin{proof}
        \textbf{Part 1:} If there would be no such $U$, then there would be a sequence $(x^j)_j \subseteq \R^n \setminus \{ x^* \}$ of critical points with $x^j \rightarrow x^*$. Since $S$ is finite, we can assume w.l.o.g.\ that $A(x^j) = S' \subseteq S$ is constant for all $j \in \N$. This implies that there is a sequence of Lagrange multipliers $(\lambda^j)_j \subseteq \Omega_{S'}$ with $\sum_{s \in S'} \lambda_s^j \nabla f_s(x^j) = 0$ for all $j \in \N$. Since $\Omega_{S'}$ is compact, $(\lambda^j)_j$ must have an accumulation point in $\Omega_{S'}$. By continuity of $\nabla f_s$, $s \in S$, and since $A(x^*) = S$, this accumulation point (embedded into $\Omega_S$ by appending zeros) must be a Lagrange multiplier of $f$ at $x^*$. By uniqueness of $\lambda^*$ (cf.\ \ref{enum:B2_2} and \ref{enum:B2_3}), this accumulation point must be $\lambda^*$. Since $\lambda_s^* > 0$ for all $s \in S$, it follows that $S' = S$, so $A(x^j) = S$ for all $j \in \N$. \\
        \textbf{Part 2:} By Lem.\ \ref{lem:quad_pos_definite}, Lem.\ \ref{lem:quad_tangent_sequence} and since all $\nabla^2 f_s$, $s \in S$, are continuous, we can assume w.l.o.g.\ that 
        \begin{align*}
            \frac{1}{2} \left( \frac{x^j - x^*}{\| x^j - x^* \|} \right)^\top 
            \left( \sum_{s \in S} \lambda^j_s \nabla^2 f_{s}(x^j) \right)
            \frac{x^j - x^*}{\| x^j - x^* \|}
            \geq \frac{\beta}{2} 
            \quad \forall j \in \N.
        \end{align*}
        \textbf{Part 3:} Second-order Taylor expansion of $\psi_j := \sum_{s \in S} \lambda_s^j f_s$ at $x^j$ with remainder $R_j$ 
        yields
        \begin{align*}
            f(x^*)
            = \psi_j(x^*)
            = f(x^j) + \frac{1}{2} (x^* - x^j)^\top \nabla^2 \psi_j(x^j)(x^* - x^j) + R_j,
        \end{align*}
        so quadratic growth of $f$ around $x^*$ implies
        \begin{align*}
            \beta
            &\leq \frac{f(x^j) - f(x^*)}{\| x^j - x^* \|^2}
            = - \frac{1}{2} \left( \frac{x^* - x^j}{\| x^j - x^* \|} \right)^\top
            \nabla^2 \psi_j(x^j)
            \frac{x^* - x^j}{\| x^j - x^* \|} - \frac{R_j}{\| x^j - x^* \|^2} \\
            &\leq -\frac{\beta}{2} - \frac{R_j}{\| x^j - x^* \|^2}
        \end{align*}
        for all $j \in \N$. By the remainder formula \eqref{eq:Taylor_remainder} for $R_j$, since all $f_s$, $s \in S$, are $\C^3$, and since a convex combination of these selection functions was expanded, the second summand on the right-hand side in the above inequality vanishes for $j \rightarrow \infty$. Thus, taking the limit leads to a contradiction.
    \end{proof}

    Combination of all results in this subsection leads to the following theorem:
    \begin{theorem} \label{thm:quad_decrease_property}
        Assume that $f$ satisfies \ref{assum:B2}. Then \eqref{eq:property_P} holds for $p = 2$.
    \end{theorem}
    \begin{proof}
        \textbf{Part 1:} Repeat the constructions of Part 1 in the proof of Thm.\ \ref{thm:sharp_decrease_property}, replacing \eqref{eq:proof_lem_sharp_decrease_property_1} with
        \begin{align} \label{eq:proof_lem_quad_decrease_property_1}
            0
            = \lim_{j \rightarrow \infty} \Lambda^2(x^j,\Delta_j)
            = \lim_{j \rightarrow \infty} \frac{f(x^j) - f(z^j)}{\Delta_j^2}.
        \end{align}
        Then repeat Part 2 in the proof of Thm.\ \ref{thm:sharp_decrease_property}, using Lem.\ \ref{lem:quad_isolated_critical_point} instead of Lem.\ \ref{lem:sharp_subgradient_bound}. \\
        \textbf{Part 2:} Second-order Taylor expansion of $\psi_j := \sum_{s \in S'} \lambda_s^j f_s$ at $z^j$ with remainder $R_j$, combined with Lem.\ \ref{lem:z_star_formula}, yields
        \begin{align*}
            &f(x^j) - f(z^j)
            \geq \psi_j(x^j) - \psi_j(z^j) \\
            &= \nabla \psi_j(z^j)^\top (x^j - z^j)
            + \frac{1}{2} (x^j - z^j)^\top \nabla^2 \psi_j(z^j) (x^j - z^j)
            + R_j \\
            &= \Delta_j \left\| \nabla \psi_j(z^j) \right\|
            + \frac{1}{2} (x^j - z^j)^\top \nabla^2 \psi_j(z^j) (x^j - z^j)
            + R_j,
        \end{align*}
        so
        \begin{equation} \label{eq:proof_lem_quad_decrease_property_3}
            \begin{aligned}
                & \Lambda^2(x^j,\Delta_j)
                = \frac{f(x^j) - f(z^j)}{\Delta_j^2} \\
                &\geq \underbrace{\frac{\left\| \sum_{s \in S'} \lambda_s^j \nabla f_s(z^j) \right\|}{\Delta_j}}_{\text{(a)}}
                + \underbrace{\frac{1}{2} \left( \tfrac{x^j - z^j}{\Delta_j} \right)^\top \left( \sum_{s \in S'} \lambda_s^j \nabla^2 f_s(z^j) \right) \tfrac{x^j - z^j}{\Delta_j}}_{\text{(b)}}
                + \underbrace{\frac{R_j}{\Delta_j^2}}_{\text{(c)}}
            \end{aligned}
        \end{equation}
        for all $j \in \N$. By the remainder formula \eqref{eq:Taylor_remainder}, the term (c) vanishes for $j \rightarrow \infty$. Since all $\nabla^2 f_s$, $s \in S$, are continuous and $\| x^j - z^j \| = \Delta_j$ for all $j \in \N$, the term (b) is bounded below for $j \rightarrow \infty$. Due to \eqref{eq:proof_lem_quad_decrease_property_1}, this means that the term (a) must be bounded above. In particular, the numerator in (a) vanishes for $j \rightarrow \infty$. Analogous to Part 1 in the proof of Lem.\ \ref{lem:quad_isolated_critical_point}, we can assume w.l.o.g.\ that $(\lambda^j)_j$ converges to $\lambda^*$ (cf.\ \ref{enum:B2_2}), which implies that $S' = A(z^j) = S$, i.e., $z^j \in \calM$, for all $j \in \N$. \\
        \textbf{Part 4:} With $\pr(x^j) \in \argmin_{y \in \calM} \| y - x^j \|$ as in Lem.\ \ref{lem:quad_orth_lin_growth}(b), we can write 
        \begin{equation} \label{eq:proof_lem_quad_decrease_property_2}
            \begin{aligned}
                \frac{f(x^j) - f(z^j)}{\Delta_j^2}
                &= \frac{f(x^j) - f(\pr(x^j))}{\Delta_j^2} + \frac{f(\pr(x^j)) - f(z^j)}{\Delta_j^2} \\
                &= \frac{f(x^j) - f(\pr(x^j))}{\| x^j - \pr(x^j) \|} \frac{\| x^j - \pr(x^j) \|}{\Delta_j^2} + \frac{f(\pr(x^j)) - f(z^j)}{\Delta_j^2}
            \end{aligned}
        \end{equation}
        for all $j \in \N$. Since $z^j \in \calM$ and $\| z^j - x^j \| = \Delta_j$, we have $\pr(x^j) \in \Bcl_{\Delta_j}(x^j)$. In particular, by definition of $z^j = z^*(x^j,\Delta_j)$ (cf.\ \eqref{eq:def_z_star_Lambda}), it holds $f(\pr(x^j)) \geq f(z^j)$, so the second summand on the right-hand side of \eqref{eq:proof_lem_quad_decrease_property_2} must be non-negative. By Lem.\ \ref{lem:quad_orth_lin_growth}(b), the first factor in the first summand is bounded below by a positive constant for all $j \in \N$ large enough. Combined with \eqref{eq:proof_lem_quad_decrease_property_1}, this means that the second factor must vanish for $j \rightarrow \infty$. \\
        \textbf{Part 5:} Consider the equality
        \begin{align*}
            \frac{x^j - z^j}{\Delta_j}
            = \frac{x^j - \pr(x^j)}{\Delta_j} + \frac{\pr(x^j) - z^j}{\Delta_j}.
        \end{align*}
        By Part 4, the first summand vanishes. By Lem.\ \ref{lem:quad_tangent_sequence}, we can assume w.l.o.g.\ that the second summand, and therefore the left-hand side, converges to an element of $\ker(D\Phi(x^*))$ with norm $1$. \\
        \textbf{Part 6:} By Lem.\ \ref{lem:quad_pos_definite} and Part 5, we can assume w.l.o.g.\ that the limit inferior of the term (b) for $j \rightarrow \infty$ in \eqref{eq:proof_lem_quad_decrease_property_3} is positive. Since the term (a) is non-negative and the term (c) vanishes for $j \rightarrow \infty$, this contradicts \eqref{eq:proof_lem_quad_decrease_property_1}, completing the proof.
    \end{proof}

\section{Numerical experiments} \label{sec:numerical_experiments}

In this section, we begin by numerically verifying the ability of \algGlobal{} to compute trust regions $\Bcl_{\Delta_j}(x^j)$ containing the minimum $x^*$ for $j$ large enough. Afterwards, we demonstrate its use for globalizing the local R-superlinear method from \cite{GU2026a} (Alg.\ 4.2) and the $k$-bundle Newton method from \cite{LW2019} (Alg.\ 2.2). Code for the reproduction of all experiments shown in this section is available at \url{https://github.com/b-gebken/higher-order-trust-region-bundle-method}. 

In all experiments, we assume that $f$ is sufficiently smooth at every point where Oracle \ref{oracle:1} is called, and use the exact analytic formulas of $f$ for the derivatives. For the parameters of \algGlobal{}, we always use
\begin{align*}
    \Delta_j = 10^{-j+1}, \quad \tau_j = 10^{-5}, \quad \sigma = 0.5, \quad c = 0.1
\end{align*}
and stop the algorithm after the outer iteration with $j = 5$ finished (i.e., the final point that is computed is $x^{5,N_5} = x^{6,0}$).
While the convergence theory requires $(\tau_j)_j$ to vanish, we found that using a ``small'' constant $\tau_j$ improved the performance. (The impact of the choice of $(\tau_j)_j$ will be discussed in Sec.\ \ref{sec:discussion_and_outlook}.)
For $q = 2$, the subproblem \eqref{eq:bar_z_epigraph} is solved via \ipopt{} \cite{WB2005} (using the Matlab interface \mexipopt{}\footnote{\url{https://github.com/ebertolazzi/mexIPOPT} (Retrieved Mar.\ 24, 2026)}). For $q = 1$, we use the maximum norm as the trust-region norm in \eqref{eq:bar_z_epigraph} (and omit the exponent $2$ in the constraint), which turns \eqref{eq:bar_z_epigraph} into a linear problem that we can solve via Matlab's \verb|linprog|. For the initial $W^1$ in Step \ref{state:global_method_approx_W} of \algGlobal{}, instead of using $W^1 = \{ x^{j,i} \}$, we keep a memory of the $100$ most recent points at which the oracle was evaluated and set $W^1$ to $\{ x^{j,i} \}$ united with those memorized points that also lie in the current trust region $\Bcl_{\Delta_j}(x^{j,i})$. (Reusing information in this way does not change the convergence analysis, since Lem.\ \ref{lem:algo_approx_W_termination} holds for any initialization of $W^1$ in \algApproxW{}.)

\subsection{Trust regions containing the minimum}

We first consider the behavior of \algGlobal{} for the well-behaved, strongly convex function (8.4) from \cite{LW2019}:
 
\begin{example} \label{example:LW2019_84_global}
    For $n, m \in \N$ and $I = \{1,\dots,m\}$, consider the strongly convex function
    \begin{align*}
        f : \R^n \rightarrow \R,
        \quad
        x \mapsto 
        \max_{i \in I}\left(  g_i^\top x + \frac{1}{2} x^\top H_i x + \frac{c_i}{24} \| x \|^4 \right)
    \end{align*}
    from \cite{LW2019}, where $c_i > 0$ for all $i \in I$, $H_i \in \R^{n \times n}$ is symmetric, pos.\ definite for all $i \in I$, and there is a subset $I' \subseteq I$ of size $\min(n+1,m)$ such that the vectors $g_i \in \R^n$, $i \in I'$, are affinely independent with $\sum_{i \in I'} \lambda_i g_i = 0$ for some strictly positive $\lambda \in \Omega_{I'}$. 
    Clearly $f$ satisfies \ref{assum:B1} with $|S| = m$. The global minimum is $x^* = 0 \in \R^n$. 
    If $n \geq m$, then the growth of $f$ around $x^*$ is quadratic ($p = 2$), and if $n < m$, then the growth is sharp ($p = 1$).
    We generate two random instances of this function, one with $n = 50$, $m = 100$, and one with $n = 50$, $m = 40$.
    By construction, these instances satisfy the assumptions of Thm.\ \ref{thm:sharp_decrease_property} and Thm.\ \ref{thm:quad_decrease_property}, respectively.
    To each instance we apply \algGlobal{} with the initial point $x^{1,0} = (1,\dots,1)^\top \in \R^n$. For the first instance we set $q = p = 1$ and for the second instance we set $q = p = 2$. The results are shown in Fig.\ \ref{fig:Example_LW2019_84_global}.
    \begin{figure}
        \centering
        \parbox[b]{0.49\textwidth}{
            \centering 
            \includegraphics[width=0.30\textwidth]{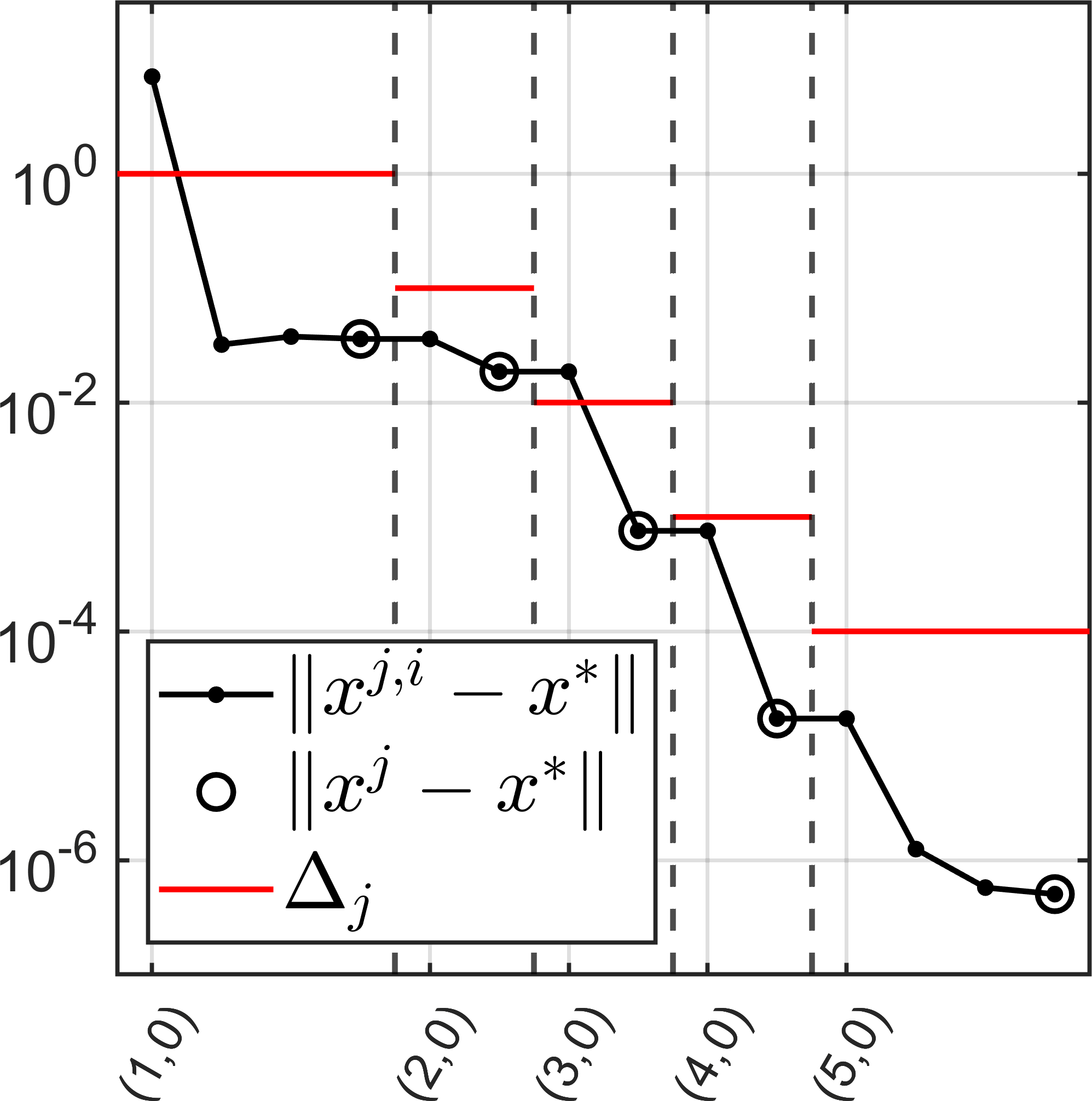}\\
        }
        \parbox[b]{0.49\textwidth}{
            \centering 
            \includegraphics[width=0.30\textwidth]{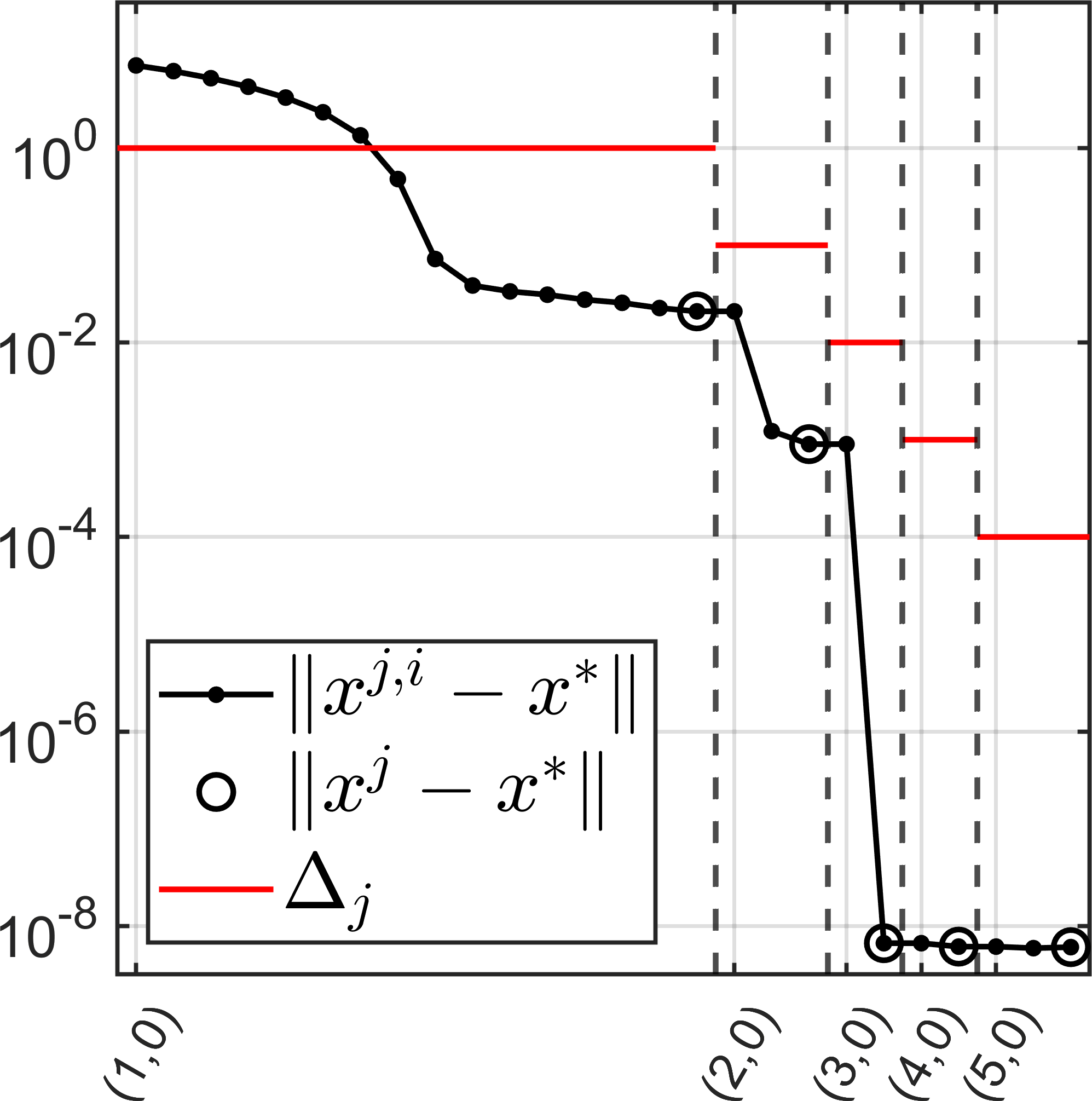}\\
        }
        \caption{The distance $\| x^{j,i} - x^* \|$ in Ex.\ \ref{example:LW2019_84_global} for $n = 50$, $m = 100$ (left) and $n = 50$, $m = 40$ (right). The horizontal axis enumerates the indices $(j,i)$ of the iterates $x^{j,i}$ in the order they are encountered in \algGlobal{}, i.e., $(1,0), (1,1), \dots, (1,N_1), (2,0), (2,1), \dots$, with the iterates corresponding to $(x^j)_j$ (cf.\ \algGlobal{}, Step \ref{state:global_method_change_j}) marked as circles. The dashed, vertical lines indicate changes of $j$ (from $(j,N_j)$ to $(j+1,0)$), and the horizontal, red lines show the trust-region radius for the corresponding $j$. (This means that visually, $x^* \in \Bcl_{\Delta_j}(x^j)$ is equivalent to $\| x^j - x^* \|$ lying below the corresponding red horizontal line.)}
        \label{fig:Example_LW2019_84_global}
    \end{figure}
    We see that for both instances, $x^* \in \Bcl_{\Delta_j}(x^j)$ holds for all $j \in \{1,\dots,5\}$.
\end{example}

For the second example we consider the nonconvex function (8.5) from \cite{LW2019}:

\begin{example} \label{example:LW2019_85_global}
    For $n, m \in \N$ and $I = \{1,\dots,m\}$, consider the nonconvex function
    \begin{align*}
        f : \R^n \rightarrow \R,
        \quad
        x \mapsto \sum_{i \in I} \left| g_i^\top x + \frac{1}{2} x^\top H_i x + \frac{c_i}{24} \| x \|^4 \right|
    \end{align*}
    from \cite{LW2019}, with $c_i$, $H_i$ and $g_i$ for $i \in I$ as in Ex.\ \ref{example:LW2019_84_global}. Then $f$ satisfies \ref{assum:B1} with $|S| = 2^m$. The global minimum is $x^* = 0 \in \R^n$. The order of growth of $f$ depends on $n$ and $m$ and is the same as in Ex.\ \ref{example:LW2019_84_global}. We again generate two random instances of this function, one with $n = 25$, $m = 100$, and one with $n = 50$, $m = 40$. In contrast to Ex.\ \ref{example:LW2019_84_global}, the assumption of Thm.\ \ref{thm:quad_decrease_property} (i.e., \ref{assum:B2}) is violated for the second instance (since $|S| > n+1$, which implies that \ref{enum:B2_3} cannot hold). Fig.\ \ref{fig:Example_LW2019_85_global} shows the results of applying \algGlobal{} to each instance with initial point $x^0 = (2,1,\dots,1)^\top \in \R^n$ and with $q = p = 1$ for the first and $q = p = 2$ for the second instance.
    \begin{figure}
        \centering
        \parbox[b]{0.49\textwidth}{
            \centering 
            \includegraphics[width=0.30\textwidth]{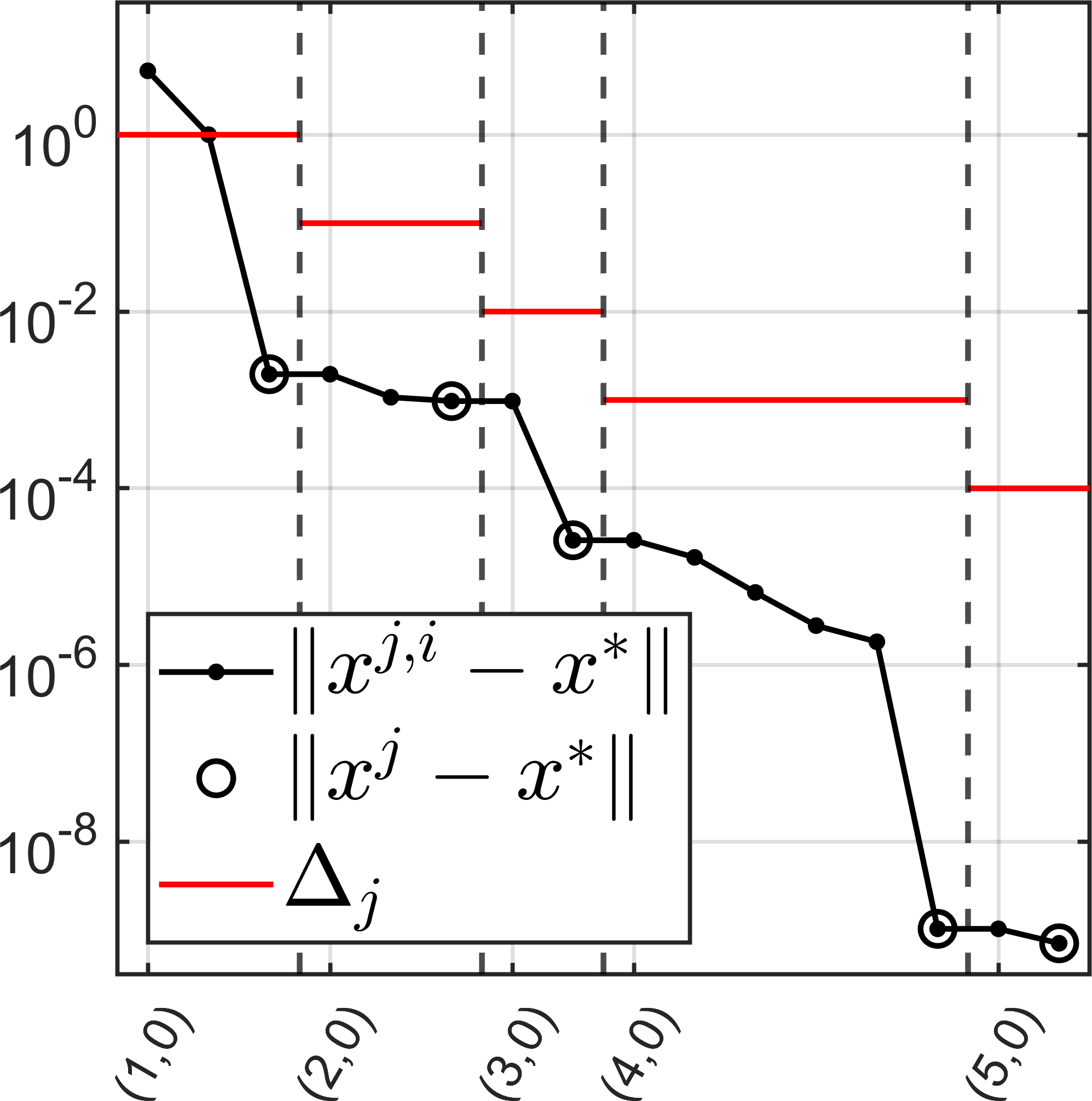}\\
        }
        \parbox[b]{0.49\textwidth}{
            \centering 
            \includegraphics[width=0.30\textwidth]{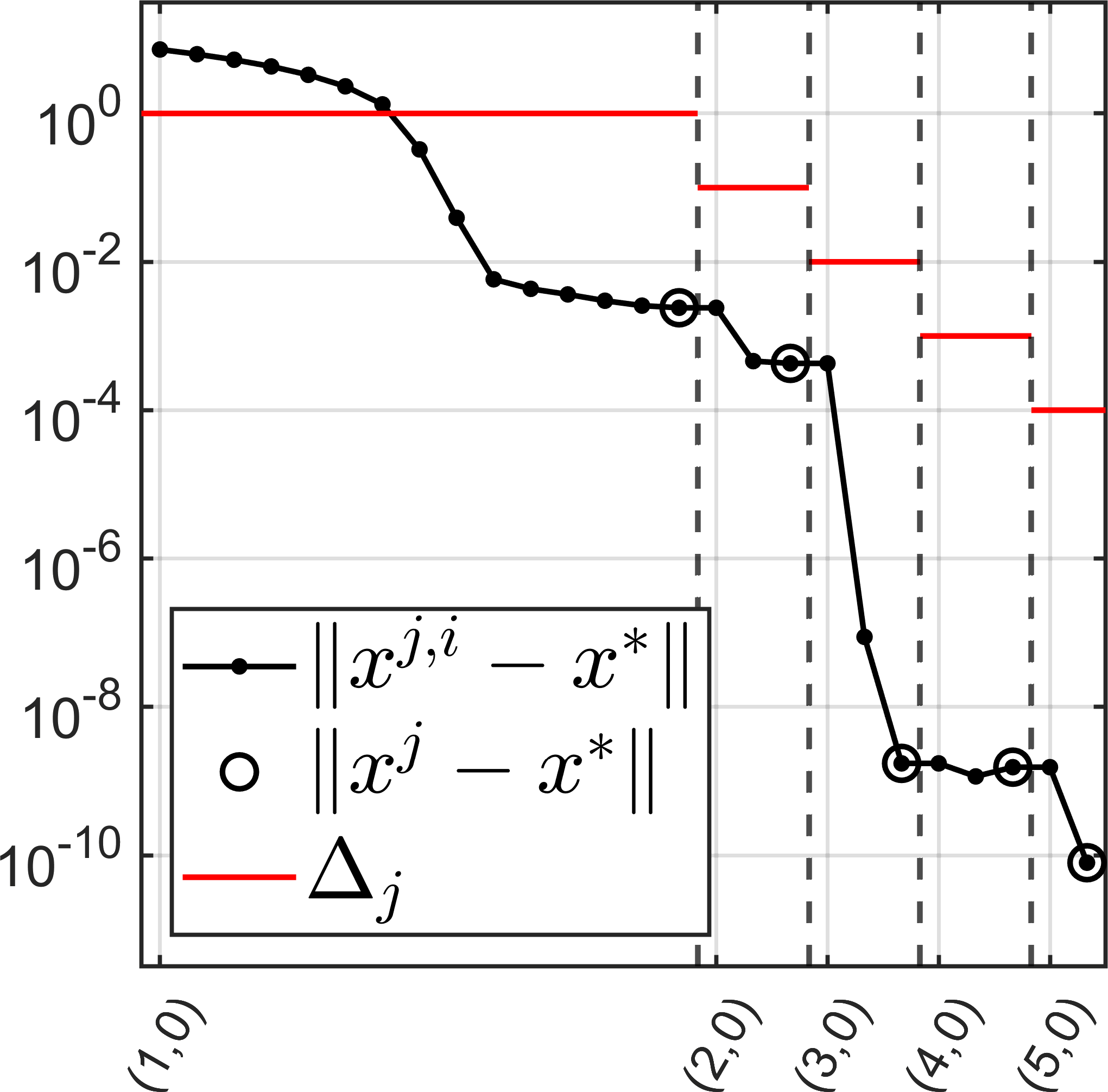}\\
        }
        \caption{The distance $\| x^{j,i} - x^* \|$ in Ex.\ \ref{example:LW2019_85_global} for $n = 25$, $m = 100$ (left) and $n = 50$, $m = 40$ (right), in the same style as in Fig.\ \ref{fig:Example_LW2019_84_global}.}
        \label{fig:Example_LW2019_85_global}
    \end{figure}
    We again see that $x^* \in \Bcl_{\Delta_j}(x^j)$ holds for all $j \in \{1,\dots,5\}$ in both instances, despite \ref{assum:B2} being violated for the instance with quadratic growth.
\end{example}

The previous example suggests that Thm.\ \ref{thm:quad_decrease_property} may also hold under weaker assumptions than \ref{assum:B2}. This is further emphasized by our final example in this subsection, where we consider a convex problem from eigenvalue optimization \cite{O1992,FN1995,LW2019}:
 
\begin{example} \label{example:LW2019_eigval_global}
    For $n, m \in \N$ and symmetric matrices $A_0, \dots, A_n \in \R^{m \times m}$, consider the function
    \begin{align*}
        f : \R^n \rightarrow \R, \quad
        x \mapsto \lambda_{\text{max}}\left( A_0 + \sum_{i = 1}^m x_i A_i \right),
    \end{align*}
    where $\lambda_{\text{max}}(A)$ denotes the largest eigenvalue of a matrix $A$. This function satisfies \ref{assum:B1} (on bounded sets) due to its convexity (cf.\ \cite{RW1998}, Thm.\ 10.33). However, in general, there is no representation with a finite $S$ (cf.\ \cite{O1992}, p.\ 89), such that $f$ is not a finite max-type function. The function $f$ is bounded below if and only if there is no $x$ for which $\sum_{i = 1}^m x_i A_i$ is positive definite (cf.\ \cite{FN1995}, p.\ 227). We generate a random instance of this problem for $n = 50$ and $m = 25$ and apply \algGlobal{} with $q = p = 2$ and initial point $x^0 = (1,\dots,1)^\top$. The result is shown in Fig.\ \ref{fig:Example_LW2019_eigval}(a).
    \begin{figure}
        \centering
        \parbox[b]{0.32\textwidth}{
            \centering 
            \includegraphics[width=0.31\textwidth]{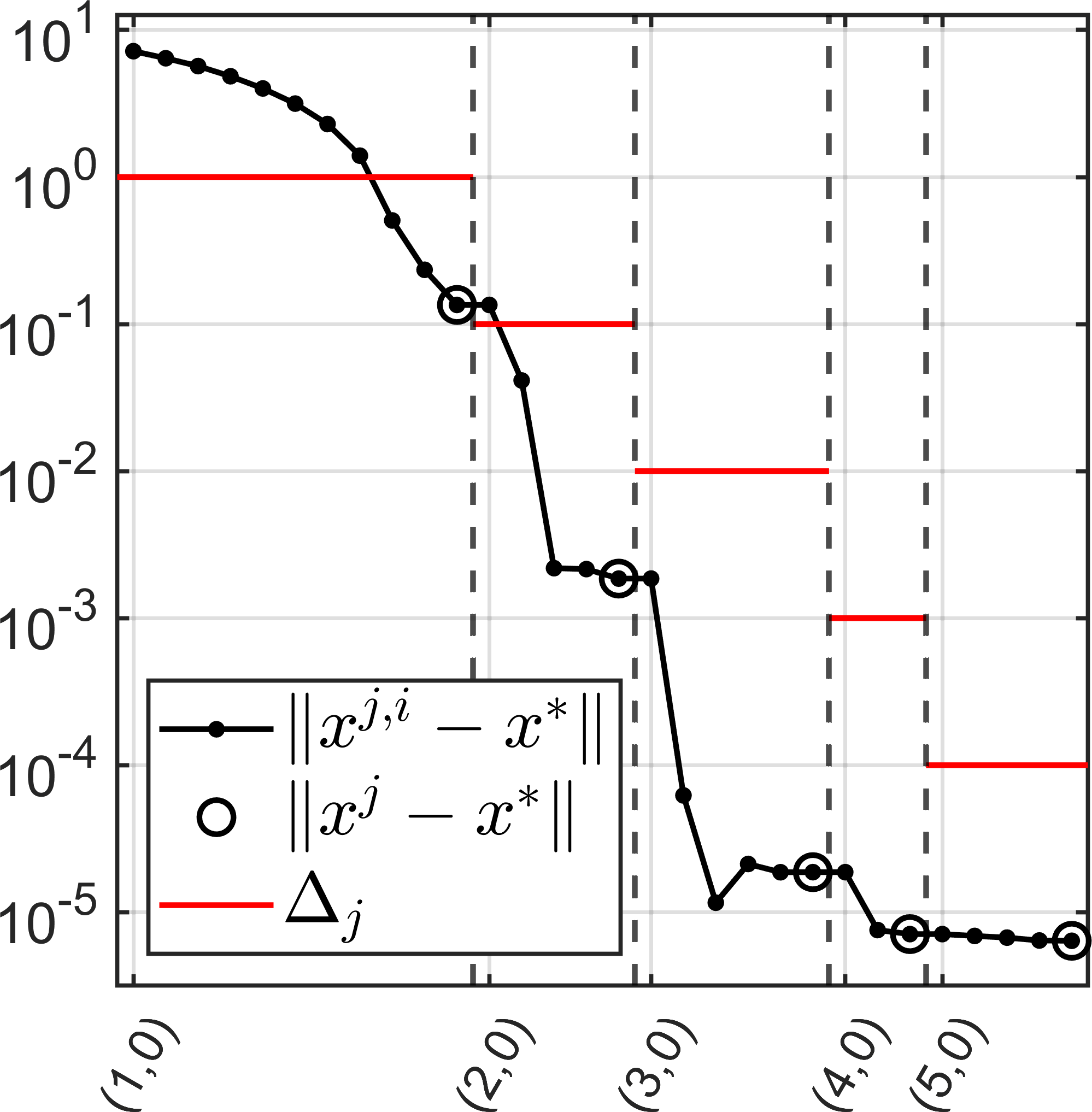}\\
            (a)
        }
        \parbox[b]{0.32\textwidth}{
            \centering 
            \includegraphics[width=0.30\textwidth]{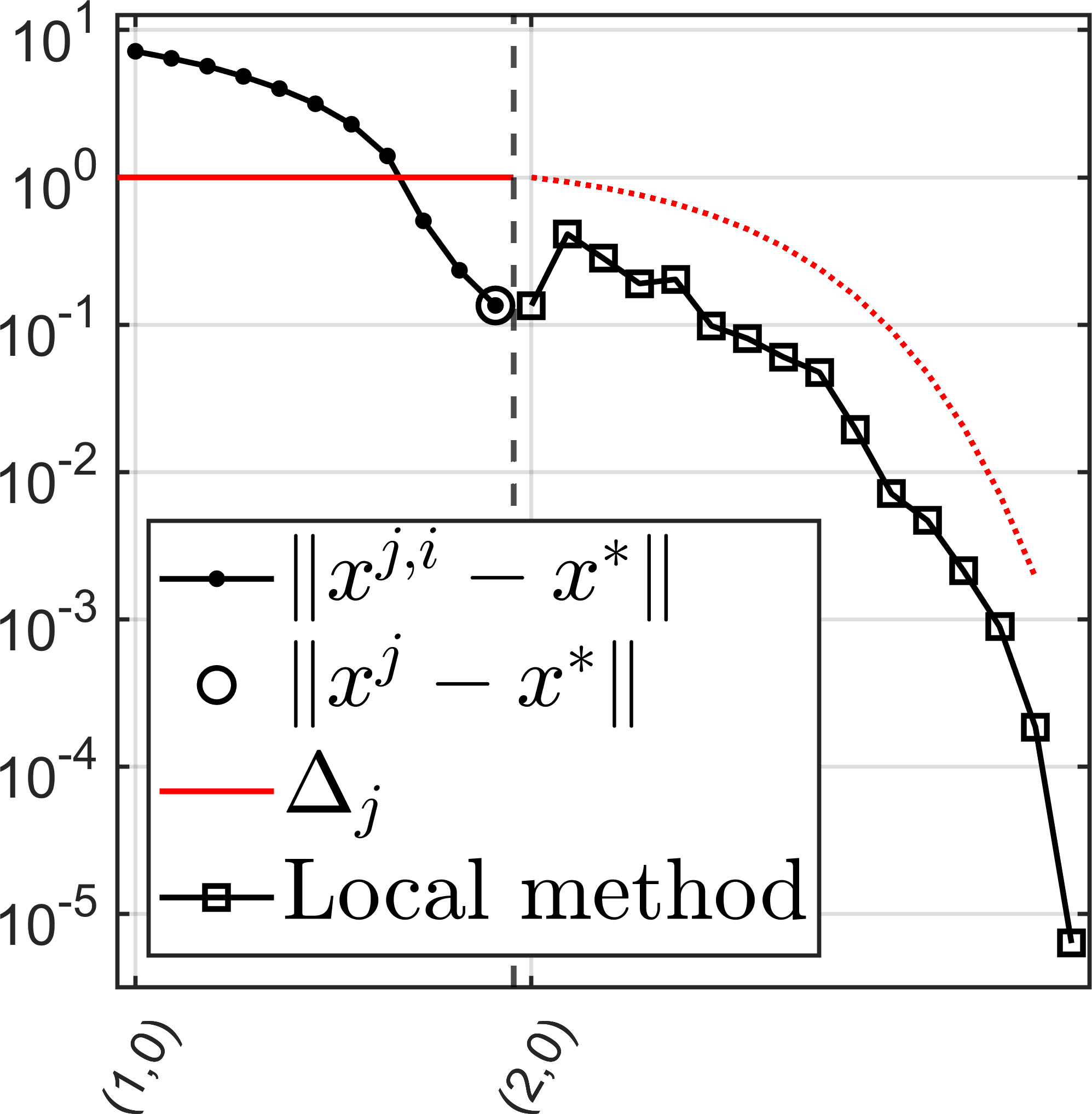}\\
            (b)
        }
        \parbox[b]{0.32\textwidth}{
            \centering 
            \includegraphics[width=0.30\textwidth]{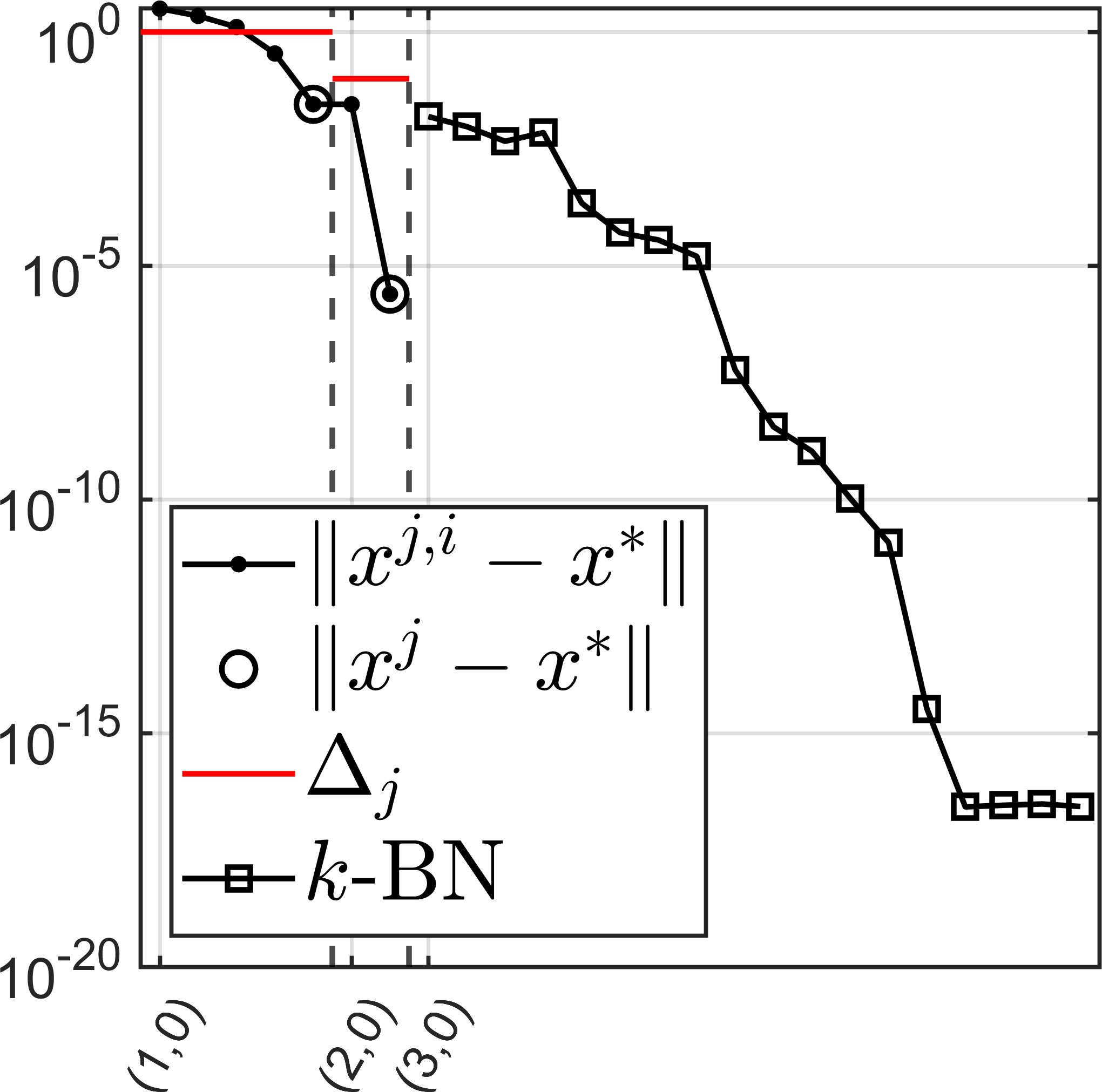}\\
            (c)
        }
        \caption{(a) The distance $\| x^{j,i} - x^* \|$ in Ex.\ \ref{example:LW2019_eigval_global}, in the same style as in Fig.\ \ref{fig:Example_LW2019_84_global}. (b) The distance $\| x^{j,i} - x^* \|$ for \algGlobal{} in Ex.\ \ref{example:LW2019_eigval_local} and the distance $\| \hat{x}^j - x^* \|$ for the sequence $(\hat{x}^j)_j$ generated by the local method from \cite{GU2026a}, with initial point $x^1$ and initial trust-region radius $\Delta_1$. The red, dotted line shows the trust-region radius of the local method. (c) Same as (b), but for Ex.\ \ref{example:LW2019_84_local} and for the sequence of $\hat{x}$ generated via the $k$-bundle Newton method ($k$-BN) from \cite{LW2019}.}
        \label{fig:Example_LW2019_eigval}
    \end{figure}
    (Since the exact value of $x^*$ is unknown, we approximated it via the \texttt{HANSO}\footnote{\url{https://cs.nyu.edu/~overton/software/hanso/} (Retrieved Mar.\ 24, 2026)} software package.) We see that $x^* \in \Bcl_{\Delta_j}(x^j)$ holds for all $j \in \{1,\dots,5\}$, despite $f$ not being a finite max-type function.
\end{example}

\begin{remark}
    Note that in all examples in this subsection, we did not only have $x^* \in \Bcl_{\Delta_j}(x^j)$ for all $j$, but even $x^* \in \Bcl_{\Delta_j}(x^{j,i})$ for $j$ large enough and all $i \in \{0,\dots,N_j\}$. However, this stronger property does not hold in general: For the toy example $f : \R \rightarrow \R$, $x \mapsto x^2$, with $x^0 = 1/2$, $\Delta_j = (1/2)^{j^2} - (1/2)^{(j+1)^2}$ and $\tau_j = 2 \frac{(1/2)^{2(j+1)^2}}{\Delta_j^2}$ for all $j \in \N$, one can show that $x^{j,0} = (1/2)^{j^2} > \Delta_j$ for all $j \in \N$. A visualization of this example can be found in the provided code.
\end{remark}

\subsection{Combination with local R-superlinear methods}

    To achieve R-superlinear convergence (of serious steps), the local method from \cite{GU2026a} (Alg.\ 4.2) requires an initial point $\hat{x}^1$ and an initial trust-region radius $\hat{\eps}^1$ such that $x^* \in \Bcl_{\hat{\eps}^1}(\hat{x}^1)$ and $\hat{\eps}^1$ is small enough (cf.\ \cite{GU2026a}, Thm.\ 4.4). If the requirements of Cor.\ \ref{cor:minimum_in_trust_region} are satisfied, then there is some $j \in \N$ such that $\hat{x}^1 = x^j$ and $\hat{\eps}^1 = \Delta_j$ from \algGlobal{} are valid choices for this initial data. Furthermore, invalid choices for the initial data can be detected by the trust-region constraint being active during the local method (see Sec.\ 5.1 in \cite{GU2026a} for details). As such, if the local method is applied after Step \ref{state:global_method_change_j} in every outer $j$-iteration of \algGlobal{}, then eventually, there is an application that converges R-superlinearly to the minimum. This is demonstrated in the following example:
     
    \begin{example} \label{example:LW2019_eigval_local}
        Consider the function from Ex.\ \ref{example:LW2019_eigval_global} and the same generated instance. Fig.\ \ref{fig:Example_LW2019_eigval}(b) shows the first outer $j$-iteration of \algGlobal{} and then an application of the local method from \cite{GU2026a} (using the same parameters as in \cite{GU2026a}, Sec.\ 6) with initial data $\hat{x}^1 = x^1$ and $\hat{\eps}_1 = \Delta_1$. During this run, the trust-region constraint in the local method is never active, resulting in the expected R-superlinear convergence.
    \end{example}
     
    As discussed in Sec.\ \ref{sec:introduction}, there are other methods that require local information of $f$ around its minimum to converge. For example, the $k$-bundle Newton method from \cite{LW2019} requires a ``full bundle'' close to the minimum, which is a set of points such that for each selection function, there is exactly one point in the bundle where this selection function is active. Considering the way in which \algApproxW{} explores the nonsmooth structure of $f$, we believe it may be possible to prove that for functions satisfying \ref{assum:B2}, \algApproxW{} computes such a full bundle when $x^* \in \Bcl_\Delta(x)$ and $\Delta$ is small enough. This would then make it possible to globalize the $k$-bundle Newton method via \algGlobal{}, as it was done for the method of \cite{GU2026a} above. We leave the theoretical analysis this requires for future research, and only demonstrate the idea in the following example:
     
    \begin{example} \label{example:LW2019_84_local}
        Consider the function from Ex.\ \ref{example:LW2019_84_global} and a random instance with $n = 10$ and $m = 5$. Fig.\ \ref{fig:Example_LW2019_eigval}(c) shows the first two outer $j$-iterations of \algGlobal{} (with no memory, i.e., $W^1 = \{ x^{j,i} \}$ is used in Step \ref{state:global_method_approx_W}) and then an application of the $k$-bundle Newton method initialized with the final bundle $W_{j,i}$ that was computed in Step \ref{state:global_method_approx_W} of \algGlobal{} in the final inner $i$-iteration before the $i$-loop was broken. (In other words, $W_{j,i}$ was computed via \algApproxW{} with $x = x^2$ and $\Delta = \Delta_2$.) The superlinear convergence that can be observed suggests that a valid initial bundle was provided for the $k$-bundle Newton method. 
    \end{example}

\section{Discussion and outlook} \label{sec:discussion_and_outlook}

We have constructed a trust-region bundle method for \lc{2} functions that is able to compute infinitely many trust regions that contain the minimum  if the property \eqref{eq:property_P} holds. We have shown that for sharp minima, property \eqref{eq:property_P} holds for all finite max-type functions, and for quadratic minima, it holds for the more well-behaved functions satisfying \ref{assum:B2}. There are several directions for future research:
\begin{itemize}
    \item For $q = 1$ and large $\Delta$, \algApproxW{} may require many iterations due to the classic stability issues of cutting-plane models described in \cite{BGLS2006}, p. 134. To avoid these issues, it may be possible to modify the models used in \algApproxW{} by introducing some type of stabilization, as discussed in \cite{BGLS2006}, Sec.\ 10.1.
    \item For the sake of simplicity, we did not try to derive a criterion for adaptively increasing the trust-region radius if the model would be trustworthy on a larger trust region. Doing so could improve the performance, especially for initial points far away from the minimum.
    \item While for the convergence of \algGlobal{}, $(\tau_j)_j$ can be any vanishing sequence, the specific choice may have a significant impact on the performance: On the one hand, the larger $\tau_j$, the earlier the inner $i$-loop is broken. For monotonically decreasing $(\Delta_j)_j$, this means that the method decreases the trust-region radius sooner. As a result, the method may perform relatively short steps even when far away from the minimum. On the other hand, the smaller $\tau_j$, the longer the algorithm may stay in the inner $i$-loop, even when the decrease in the objective value becomes small. As a result, the method may perform many steps with a trust region that is too large. To avoid both behaviors, a theoretical analysis on the impact of $(\tau_j)_j$ on the performance is required.
    \item It may be possible to prove that \eqref{eq:property_P} holds for arbitrary $p$ by replacing the second-order Taylor expansions in Lem.\ \ref{lem:quad_pos_definite}, Lem.\ \ref{lem:quad_isolated_critical_point}, and Thm.\ \ref{thm:quad_decrease_property} by $p$-order expansions and then carefully estimating the individual terms.
    \item The generalization of Sec.\ \ref{sec:property_P_for_max_type} to \lc{2} functions with infinite $S$ does not appear to be straightforward. In Part 1 of the proofs of both Thm.\ \ref{thm:sharp_decrease_property} and Thm.\ \ref{thm:quad_decrease_property}, we were only able to assume that the active set along $(z^j)_j$ is constant since $S$ was finite. For infinite $S$ this is no longer possible, which means that in all other parts of both proofs, the index set $S'$ depends on iteration index $j$, making the analysis significantly more challenging. 
\end{itemize}

\vspace{10pt}

\setlength{\bibsep}{0pt plus 0.3ex}

\noindent \textbf{Acknowledgements.} \quad This research was funded by Deutsche Forschungsgemeinschaft (DFG, German Research Foundation) – Projektnummer 545166481.

\bibliography{references}

\end{document}